\documentclass[10pt, letterpaper,reqno]{amsart}
\usepackage[margin=0.945in]{geometry}
\usepackage{amsfonts}
\usepackage{amsmath, amssymb,mathtools}
\usepackage{graphicx}
\usepackage[font=small,labelfont=bf]{caption}
\usepackage{epstopdf}
\usepackage{xcolor}
\usepackage{amsthm}
\usepackage{float}
\usepackage{enumitem}
\usepackage{pgfplots}
\usepackage{listings}
\usepackage{longtable}
\usepackage{mathrsfs}
\usepackage{dsfont}
\usepackage{mathtools}
\usepackage{dsfont} 
\usepackage{tipa}
\usepackage{xfrac} 
\usepackage{stmaryrd}
\usepackage[normalem]{ulem}
\usepackage[pdfpagelabels,hyperindex]{hyperref}
\hypersetup{linkbordercolor=green}

\newcommand*{\rom}[1]{\expandafter\@slowromancap\romannumeral #1@}
\DeclareMathAlphabet{\mathpzc}{OT1}{pzc}{m}{it}
 \colorlet{lgray}{white!80!black}
\colorlet{lred}{white!85!red}
\colorlet{lgreen}{white!60!green}
\colorlet{dgreen}{black!30!green}
\colorlet{lpurple}{white!60!purple}
\colorlet{lblue}{white!60!blue}
\definecolor{green}{rgb}{0.1,0.8,0.1}
\definecolor{yellow}{rgb}{1.0,0.85,0.25}
\definecolor{purple}{rgb}{1.0, 0, 1.0}
\definecolor{blue}{rgb}{0, 0, 1.0} 
\tikzstyle{unfused}=[lgray, line width=1.5pt, ->]
\tikzstyle{fused}=[lgray, line width=4pt, ->]
\tikzstyle{dual}=[black, line width=1pt, dashed]
\tikzstyle{lightdual}=[black, line width=0.5pt, dashed]
\tikzstyle{cut}=[black, line width=1.0pt]

\theoremstyle{plain}
\newtheorem{theorem}{Theorem}[section]
\newtheorem{lemma}[theorem]{Lemma}
\newtheorem{proposition}[theorem]{Proposition}
\newtheorem{corollary}[theorem]{Corollary}

\theoremstyle{definition}
\newtheorem{definition}[theorem]{Definition}
\newtheorem{remark}[theorem]{Remark}

\numberwithin{equation}{section}
\pgfplotsset{compat=1.9} 
\DeclareFontFamily{U}{mathx}{}
\DeclareFontShape{U}{mathx}{m}{n}{<-> mathx10}{}
\DeclareSymbolFont{mathx}{U}{mathx}{m}{n}
\DeclareMathAccent{\widehat}{0}{mathx}{"70}
\DeclareMathAccent{\widecheck}{0}{mathx}{"71}  
\newcommand\shorttitle[1]{\renewcommand\@shorttitle{#1}}
\newcommand{\upperRomannumeral}[1]{\uppercase\expandafter{\romannumeral#1}} 
\def \be {\begin{equation}}
\def \ee  {\end{equation}}
\def \bestar {\begin{equation*}}
\def \eestar  {\end{equation*}}
\def \lb {\left(}
\def \rb {\right)} 
\def \llb {\llbracket}
\def \rrb {\rrbracket}
\newcommand{\intint}[1]{\llbracket#1\rrbracket}

\def \EE {\mathbb{E}}
\def \RR {\mathbb{R}}
\def \ep {\varepsilon} 
\def \ZZ {\mathbb{Z}} 
\def \PP {\mathbb{P}} 
\def \d {\mathsf{d}}
\def \a {\mathsf{a}}
\def \b {\mathsf{b}}
\def \c {\mathsf{c}}
\def \d {\mathsf{d}}
\def \th {\theta}
\def \de {\delta}
\def \Iab {I^{(\a,\b)}}
\def \Iabth {I^{(\a,\b)}_{\th_1,\dots,\th_{d+1}}}
\def \bl{\boldsymbol{\lambda}}
\def \la {\lambda}
\def \one {\mathbf{1}} 

\def \D {\mathbf{D}}
\def \E {\mathbf{E}} 
\def \ll{\langle}
\def \rr{\rangle}

\DeclareMathOperator{\wt}{wt}
\DeclareMathOperator{\cone}{Cone}
\allowdisplaybreaks
\usepgflibrary{fpu}
\makeatletter
\allowdisplaybreaks
\let\NAT@parse\undefined
\makeatother
\author{Milind Hegde}
\address
{
Milind Hegde, 
Columbia University, USA  
}
\email{mh4259@columbia.edu} 
\author{Zongrui Yang}
\address
{
Zongrui Yang, 
Columbia University, USA 
}
\email{zy2417@columbia.edu}

\title[Large deviation principle for the stationary measures of open ASEP]{Large deviation principle for the stationary measures of\\open asymmetric simple exclusion processes}
 
\begin{document} 

\begin{abstract}
We consider the stationary measure of the asymmetric simple exclusion process (ASEP) on a finite interval in $\ZZ$ with open boundaries. Fixing all the jump rates and letting the system size approach infinity, the height profile of such a sequence of stationary measures satisfies a large deviation principle (LDP), whose rate function was predicted in the physics work \cite{derrida2003exact}. In this paper, we provide the first rigorous proof of the large deviation principle in the ``fan region'' part of the phase diagram. Our proof relies on two key ingredients: a two-layer expression of the stationary measure of open ASEP, arising from the Enaud--Derrida representation \cite{enaud2004large} of the matrix product ansatz, and the large deviation principle of the open totally asymmetric simple exclusion process (TASEP) recently established in \cite{bryc2024two}. 
\end{abstract}

 \setcounter{tocdepth}{1}
\maketitle 

\tableofcontents

\section{Introduction and main results}
\subsection{Preface}\label{subsec:preface}
The open asymmetric simple exclusion process (ASEP) is a paradigmatic model for non-equilibrium systems with open boundaries and, asymptotically, for Kardar–Parisi–Zhang (KPZ) universality. Over the past five decades, extensive studies have been dedicated to understanding its stationary measure through the ``matrix product ansatz'' approach. This method has enabled the derivation of a wide range of asymptotic behaviors of the stationary measure of open ASEP, including phase diagrams, density profiles, correlation functions, limit fluctuations, and large deviations. We refer the reader to \cite{liggett1999stochastic, corwin2022some, blythe2007nonequilibrium, corteel2011tableaux} for a selection of surveys in statistical physics, probability, and combinatorics.

The large deviation principle (LDP) for the height profile of the stationary measure of open ASEP has been calculated in the physics literature \cite{derrida2002exact,derrida2003exact}, using the matrix product ansatz and the additivity principle. We refer the reader to \cite{derrida2006matrix,derrida2007non} for expositions on this topic; Section~\ref{s.related work} also contains further discussion of related work. Rigorously establishing the large deviation principle for open ASEP has been a long-standing open problem. Recently, the LDP for the open totally asymmetric simple exclusion process (TASEP) was proved in \cite{bryc2024two}, utilizing a representation developed therein in terms of two-layer Gibbs measures. The present work focuses on the LDP for open ASEP.

The phase diagram of the open ASEP consists of the fan region and the shock region. In this paper, we study the fan region, where we rigorously confirm the large deviation principle proposed in \cite{derrida2002exact, derrida2003exact}. Our approach relies on introducing a two-layer representation of the stationary measure of the open ASEP in the fan region, covering the full range of parameter choices. This representation originates in the Enaud–Derrida formulation of the matrix product ansatz \cite{enaud2004large} and has been previously discussed in \cite{bryc2024two,bryc2024stationary,derrida2004asymmetric,barraquand2023stationary,brak2006combinatorial,nestoridi2024approximating} for particular choices of parameters. We then develop key estimates for the two-layer measures, enabling a comparison of the two-layer weights for open ASEP with those for open TASEP. This is achieved using monotonicity properties of the open TASEP two-layer Gibbs measure along with resampling arguments, which represent a novel aspect of our approach compared to previous studies on LDPs for open ASEP/TASEP. Based on this comparison, we establish that open ASEP satisfies an LDP with a rate function that depends only on the effective densities at the two boundaries; in particular, it does not depend on the asymmetry parameter. Consequently, the rate function must coincide with the one in the case of open TASEP, which was identified in \cite{bryc2024two}.  Some more details are discussed in Section~\ref{sec:method of proof}.

Our approach is fairly robust. An upcoming work 
of the second-named author with Dominik Schmid will provide a different version of a two-layer representation of the stationary measure of open ASEP in the shock region under Liggett's condition, which is based on \cite{bryc2024stationary}. Utilizing this representation, the LDP in the shock region can also be demonstrated using a similar argument, as will be shown in a future version.

\subsection{Definition of the model}\label{subsec:model and results}
The open ASEP is a continuous-time irreducible Markov process on the state space $\{0,1\}^N$ with parameters
\be\label{eq:conditions open ASEP}
\alpha,\beta>0,\quad \gamma,\de\geq 0,\quad\text{and}\quad  0\leq q<1,
\ee  
which models the evolution of particles on the lattice $\left\{1,\dots,N\right\}$; here $q$ is the asymmetry parameter mentioned above. The state of the system is represented by a configuration $(\tau_1,\ldots,\tau_N)\in\{0,1\}^N$, where $\tau_i$ indicates whether the $i$\textsuperscript{th} site is occupied or empty. In the bulk of the system, each particle attempts moves at random: to the left with rate $q$, and to the right with rate $1$, independently of all other particles. Here, at rate $q$ or $1$ means at random times given by a Poisson clock with the corresponding rate. At the two boundaries are reservoirs with infinitely many particles. 
At the left boundary, particles attempt to enter the system from the reservoir at random with rate $\alpha$ and attempt to exit at random with rate $\gamma$. At the right boundary, particles attempt to enter at random with rate $\delta$ and attempt to exit at random with rate $\beta$. All these attempts fail if the target site is already occupied, and succeed if not (attempts to exit the system into the reservoirs always succeed). See Figure \ref{fig:openASEP} for an illustration. Since $q < 1$, particles move in an asymmetric way, with a higher rate to the right than to the left. In the special case $q = \gamma=\delta=0$, particles move only to the right and the model is known as the open totally asymmetric simple exclusion process (TASEP).
 
\begin{figure}[t]
\centering
 \begin{tikzpicture}[scale=0.8]
\draw[thick] (-1.2, 0) circle(1.2);
\draw (-1.2,0) node{reservoir};
\draw[thick] (0, 0) -- (12, 0);
\foreach \x in {1, ..., 12} {
	\draw[gray] (\x, 0.15) -- (\x, -0.15);
}
\draw[thick] (13.2,0) circle(1.2);
\draw(13.2,0) node{reservoir};
\fill[thick] (2, 0) circle(0.2);
\fill[thick] (5, 0) circle(0.2);
\fill[thick] (7, 0) circle(0.2);
\fill[thick] (10, 0) circle(0.2);
\draw[thick, ->] (2, 0.3)  to[bend left] node[midway, above]{$1$} (3, 0.3);
\draw[thick, ->] (5, 0.3)  to[bend right] node[midway, above]{$q$} (4, 0.3);
\draw[thick, ->] (7, 0.3) to[bend left] node[midway, above]{$1$} (8, 0.3);
\draw[thick, ->] (10.05, 0.3) to[bend left] node[midway, above]{$1$} (11, 0.3);
\draw[thick, ->] (9.95, 0.3) to[bend right] node[midway, above]{$q$} (9, 0.3);
\draw[thick, ->] (0, 0.5) to[bend left] node[midway, above]{$\alpha$} (0.9, 0.4);
\draw[thick, <-] (0, -0.5) to[bend right] node[midway, below]{$\gamma$} (0.9, -0.4);
\draw[thick, ->] (12, -0.4) to[bend left] node[midway, below]{$\delta$} (11, -0.5);
\draw[thick, <-] (12, 0.4) to[bend right] node[midway, above]{$\beta$} (11.1, 0.5);
\node[gray] at (1,-0.4) {\tiny $1$};\node[gray] at (2,-0.4) {\tiny $2$};\node[gray] at (3,-0.4) {\tiny $3$};\node[gray] at (4,-0.4) {\tiny $4$};\node[gray] at (5,-0.4) {\tiny\dots};\node[gray] at (11,-0.4) {\tiny $n$};
\end{tikzpicture} 
\caption{Jump rates in the open ASEP.}
\label{fig:openASEP}
\end{figure}
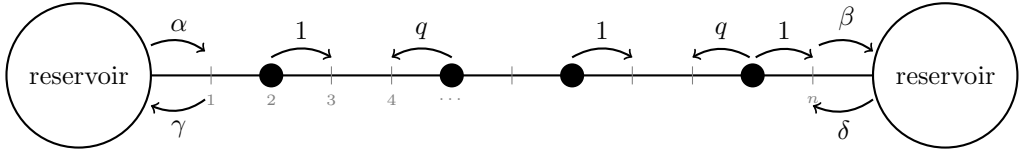

We will work with a re-parameterization of the open ASEP system by $\a,\b,\c,\d$ and $q$, where
\be\label{eq:defining ABCD}
\a=\phi_+(\alpha,\gamma),\quad\b=\phi_+(\beta,\de),\quad \c=\phi_-(\alpha,\gamma), \quad \d=\phi_-(\beta,\de),
\ee 
and
\be \label{eq:phi}
\phi_{\pm}(x,y)=\frac{1}{2x}\lb1-q-x+y\pm\sqrt{(1-q-x+y)^2+4xy}\rb, \quad \mbox{for }\; x>0\mbox{ and }y\geq 0.
\ee  
One can check that \eqref{eq:defining ABCD} gives a bijection between \eqref{eq:conditions open ASEP} and 
\be\label{eq:conditions qABCD}
\a,\b\geq0,\quad -1<\c,\d\leq 0,\quad 0\leq q<1.
\ee 

We will assume \eqref{eq:conditions open ASEP}, and consequently  \eqref{eq:conditions qABCD}, throughout the paper. 
The quantities $1/(1+\a)$ and $\b/(1+\b)$ defined by the parameters above have nice
physical interpretations as the ``effective densities'' near the left and right boundaries of the system.
In the special case of the open TASEP, i.e. $q = \gamma=\delta=0$, we have
\[
\a=(1-\alpha)/\alpha\geq0,\quad\b=(1-\beta)/\beta\geq0,\quad \c=\d=q=0.
\]

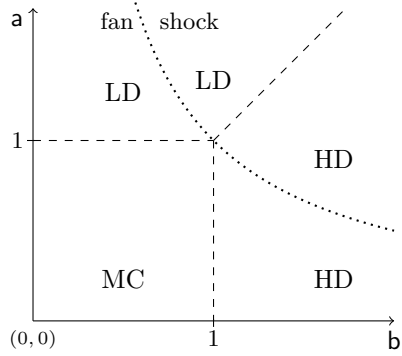
\begin{figure}[ht]
    \centering
    \begin{tikzpicture}[scale=0.8]
 \draw[scale = 1,domain=6.7:11,smooth,variable=\x,dotted,thick] plot ({\x},{1/((\x-7)*1/3+2/3)*3+5});
 \draw[->] (5,5) to (5,10.2);
 \draw[->] (5.,5) to (11,5);
   \draw[dashed] (5,8) to (8,8);
  \draw[dashed] (8,8) to (8,5);
   \draw[dashed] (8,8) to (10.2,10.2);
   \node [left] at (5,8) {\small$1$};
   \node[below] at (8,5) {\small $1$};
     \node [below] at (11,5) {$\b$};
   \node [left] at (5,10) {$\a$};
 \draw[dashed] (8,4.9) to (8,5.1);
  \draw[dashed] (4.9,8) to (5.1,8);
 \node [below] at (5,5) {\scriptsize$(0,0)$};
    \node [above] at (6.5,8.5) {LD};\node [above] at (8,8.7) {LD};
    \node [below] at (10,6) {HD};  \node [below] at (10,8) {HD}; 
 \node [below] at (6.5,6) {MC}; 
 \node at (6.4,10) {\small{fan}};
 \node at (7.6,10) {\small{shock}};
 \end{tikzpicture} 
    \caption{Phase diagram for the open ASEP stationary measures. LD, HD, MC respectively
stand for the low density, high density and maximal current phases.}
    \label{fig:phase}
\end{figure}

It is known since \cite{derrida1993exact} that as the system size $N\rightarrow\infty$, the asymptotic behavior of open ASEP is governed by the parameters $\a$ and $\b$, which exhibits a phase diagram (Figure \ref{fig:phase}) involving three phases:
\begin{enumerate}
    \item [$\bullet$] (maximal current phase) $\a<1$, $\b<1$,
    \item [$\bullet$] (high density phase) $\b>1$, $\b>\a$, and
    \item [$\bullet$] (low density phase) $\a>1$, $\a>\b$.
\end{enumerate} 

The phase diagram can also be divided into two regions which evince quite different behavior. They are defined by \cite{derrida2002exact,derrida2003exact}:
\begin{enumerate}
    \item [$\bullet$] (fan region) $\a\b<1$ and
    \item [$\bullet$] (shock region) $\a\b>1$.
\end{enumerate}

We denote by $\mu_N=\mu_N^{(q,\a,\b,\c,\d)}$ the unique stationary measure of open ASEP, which is a probability measure on $(\tau_1,\dots,\tau_N)\in\{0,1\}^N$, where, as mentioned earlier, $\tau_i$ is the occupation variable of site $i$, for $i=1,\dots,N$. 
The (re-scaled) height profile  of open ASEP is a continuous function $h_N:[0,1]\rightarrow[0,1]$ defined by 
\[h_N\lb\frac{k}{N}\rb:=\frac{1}{N}\sum_{i=1}^k\tau_i,\quad\text{for } k=0,\dots,N\]
and by using linear interpolation to extend its definition to $[0,1]$. One can observe that the height profile $h_N$ is a (non-strictly) increasing, 1-Lipschitz continuous function.

\subsection{The large deviation principle}
\begin{definition}\label{def:LDP}
Let $\mathcal{X}$ be a Polish space. 
A good rate function is a lower semi-continuous function $I:\mathcal{X}\to[0,\infty]$ with compact level sets, i.e., $I^{-1}[0,a]$ is a compact subset of $\mathcal{X}$ for every $a\geq 0$. 
Consider a sequence of probability spaces
$(\Omega_N,\PP_N)$ and a family of random variables $X_N:\Omega_N\to \mathcal{X}$, $N=1,2,\dots$.
The sequence $\{X_N\}_{N=1}^{\infty}$ satisfies the large deviation principle (LDP) with good rate function $I:\mathcal{X}\to[0,\infty]$ if 
\begin{enumerate}
\item [(1)] For any closed set $C\subset \mathcal{X}$,
    \begin{equation}\label{eq:def of LDP second ineq}
    \limsup_{N\to\infty}\frac{1}{N} \log \PP_N (X_N\in  C)\leq -\inf_{x\in C} I(x).
    \end{equation}
    \item [(2)] For any open set $U\subset \mathcal{X}$,
    \begin{equation}\label{eq:def of LDP first ineq}
    \liminf_{N\to\infty}\frac{1}{N} \log \PP_N (X_N\in  U)\geq -\inf_{x\in U} I(x).
    \end{equation}
\end{enumerate} 
\end{definition}

We next introduce our main theorem, which is proved in Section~\ref{sec:proof of LDP}. We will denote by $C_0\lb[0,1],\RR\rb$ the space of continuous functions $f:[0,1]\rightarrow\RR$ with $f(0)=0$, equipped with the supremum norm. We will also denote by $\mathcal{AC}_0\subset C_0\lb[0,1],\RR\rb$ the subset of absolutely continuous functions.

\begin{theorem}\label{thm:LDP main thm}
Consider the open ASEP in the fan region $\a\b<1$. 
    Under the stationary measures $\mu_N^{(q,\a,\b,\c,\d)}$ of open ASEP, the sequence of height profiles $\{h_N\}_{N=1}^{\infty}\subset C_0\lb[0,1],\RR\rb$ satisfies the large deviation principle with the good rate function $\Iab:C_0\lb[0,1],\RR\rb\rightarrow[0,\infty]$ defined as follows.

    When $f\in \mathcal{A}C_0\lb[0,1],\RR\rb$ satisfies $0\leq f'\leq 1$ almost everywhere,
    \be\label{eq:rate fan}\Iab(f):=\int_0^1 H \lb f'(x)\rb\d x + \int_0^1 \lb\widetilde{f'}(x) \log  G_*(x)+\lb1-\widetilde{f'}\rb\log(1- G_*(x))\rb\d x-\log (J(\a,\b)),\ee
where $H(x):=x \log x +(1-x)\log(1-x)$ if $x\in[0,1]$ and $H(x)=\infty$ otherwise;  
\be\label{eq:J}J(\a,\b):=\begin{cases}
    \a/(1+\a)^2 &\text{ if  }\a>1,\\
    \b/(1+\b)^2 &\text{ if  }\b>1,\\
    1/4 &\text{ if  }\a,\b\leq 1;
\end{cases}\ee
$G_*(x):=\min\lb\max\lb\widetilde{f'}(x),\a/(1+\a)\rb,1/(1+\b)\rb$; and $\widetilde{f}$ is the convex envelope of $f$, i.e., the largest convex function below $f$. For any other $f\in C_0\lb[0,1],\RR\rb$ we define $\Iab(f)=\infty$.
\end{theorem}
\begin{remark}
    The above large deviation principle for the stationary measure of open ASEP was first postulated in the physics works \cite{derrida2002exact,derrida2003exact} by B. Derrida, J. L. Lebowitz and E. R. Speer. Recently it was rigorously shown in \cite{bryc2024two} by W. Bryc and P. Zatitskii in the case of open TASEP. The contribution of this paper is to rigorously demonstrate this result in the fan region $\a\b<1$ for general open ASEP.
\end{remark}
\begin{remark}\label{rmk:alternative formula rate function}
    An alternative formula for the rate function $\Iab:C_0\lb[0,1],\RR\rb\rightarrow[0,\infty]$ was provided in the recent work \cite[Theorem 3.4]{bryc2024two}.
    We will utilize this formula so we record it as follows.
    We have
    \[
    \Iab(f)=\inf_{ \substack{ g\in\mathcal{AC}_0, \\ 0\leq g'\leq 1 \mbox{ a.s.} }}\Iab(f,g)
    \]
    if $f\in\mathcal{AC}_0$ and $0\leq f'\leq1$ almost surely, and $I(f)=\infty$ otherwise, where 
    \[ 
        \Iab(f,g)= \int_0^1 \bigl(H(f'(x)) +H(g'(x))\bigr) \d x
    + \log(\a\b) \min_{0\leq x\leq 1}(f(x)-g(x)) -\log (\b)(f(1)-g(1)) - \log (J(\a,\b)).
    \]
\end{remark}
\begin{remark}
    Along the boundary curve $\a\b=1$ of the fan region, the stationary measure $\mu_N$ of open ASEP becomes a product of i.i.d. Bernoulli random variables $\tau_1,\dots,\tau_N$ with mean $1/(1+\a)=\b/(1+\b)$; see, for example, \cite[Appendix A]{enaud2004large}. The large deviation principle therein can be obtained using standard results.
\end{remark}

\subsection{Related works}\label{s.related work}  
One of the most important methods for studying the stationary measures of simple exclusion processes on a finite lattice is the matrix product ansatz (MPA). This method was introduced by Liggett \cite{liggett1975ergodic} in an implicit form and by Derrida, Evans, Hakim, and Pasquier \cite{derrida1993exact} for the open TASEP; see \cite{blythe2000exact} for its extension to open ASEP. The matrix product ansatz has inspired numerous studies on various types of asymptotics of open ASEP. For a non-exhaustive list of works in physics, see \cite{derrida1993exact,schutz1993phase,essler1996representations,mallick1997finite,derrida2002exact,derrida2003exact,sasamoto2000density,uchiyama2004asymmetric,uchiyama2005correlation} and the references in \cite{blythe2007nonequilibrium,corwin2022some}.
It is known from \cite{uchiyama2004asymmetric} (see also \cite{corteel2011tableaux}) that the matrix product ansatz has a representation closely related to the Askey–Wilson orthogonal polynomials \cite{askey1985some}. This connection has led to a powerful method for studying the stationary measure of open ASEP \cite{uchiyama2004asymmetric,bryc2010askey,bryc2017asymmetric,bryc2019matrix,wang2023askey}.
 
The stationary measure of open ASEP is closely connected to many combinatorial structures, including staircase tableaux, Motzkin paths, and two-layer measures. These structures arise from various representations of the matrix product ansatz. See the survey papers \cite{corteel2011matrix,wood2020combinatorial} and references therein. The two-layer Gibbs measure was recently introduced by \cite{barraquand2024stationary} to characterize the stationary measures of certain integrable models with two open boundaries. This representation was later developed for the open TASEP in \cite{bryc2024two} and the open ASEP under Liggett's condition in \cite{bryc2024stationary}. See also \cite{bryc2024limitCoex,barraquand2024integral} for related works. This two-layer reweighted random walk representation had appeared earlier for the open ASEP in the fan region \cite{barraquand2023stationary} under Liggett's condition; see \cite{derrida2004asymmetric} for the open TASEP case. Another bi-colored Motzkin path (or equivalently, two-layer measures) representation for the $\gamma=\delta=0$ case is provided in \cite{brak2006combinatorial}; see \cite{nestoridi2024approximating} for discussions and applications. A different representation of the stationary measure of TASEP was studied in \cite[Section 3.2]{duchi2005combinatorial}, which has a Markov evolution structure of  two layers.

The large deviation principle (LDP) is an important aspect of the stationary measure of the open ASEP and has been the focus of extensive study. In the physics literature, the LDP for the open symmetric simple exclusion process (SSEP) was calculated in \cite{derrida2001free,derrida2002large} using the matrix product ansatz by directly computing the probability of a macroscopic profile through the summation of configuration probabilities. For the open ASEP, the LDP was calculated in \cite{derrida2002exact,derrida2003exact} for the case $\gamma=\delta=0$ and $0\leq\alpha,\beta\leq1-q$. This calculation is more complex and was carried out using the ``additivity property" which was obtained through the matrix product ansatz. For the weakly asymmetric limit (i.e. $q\rightarrow1$ as $q=1-\lambda/N$), the LDP was calculated in \cite{enaud2004large}. We refer the reader to \cite{derrida2006matrix,derrida2007non} for introductory surveys on these topics. For investigations on the LDP for the current of the stationary measure of open ASEP, see physics works \cite{bodineau2005current,de2011large,lazarescu2011exact,lazarescu2013exact,gorissen2012exact}, survey papers \cite{lazarescu2015physicist,mallick2015exclusion} and more references therein.

In the mathematical literature, the LDP of the total number of particles for open ASEP was obtained in \cite{bryc2017asymmetric} using Askey--Wilson methods. 
This observable is a marginal of the full height profile.
It is not clear whether the Askey--Wilson method is suitable for deriving the LDP of the height profile, due to a technical constraint in the Laplace transform coefficients. On the other hand, the two-layer representation of open ASEP is suitable for studying the LDP of the height profile. The LDP for the open TASEP stationary measure was obtained in \cite{bryc2024two} via the two-layer representation, using directly Mogulskii’s theorem in LDP theory. An alternative formula for the rate function was also obtained in \cite{bryc2024two}; see Remark \ref{rmk:alternative formula rate function}.

In the present work, we derive the LDP of the open ASEP in the fan region using a different version of the two-layer representation compared to \cite{bryc2024two}, where the two layers are both fixed to start at zero and are not ordered. In particular, in our representation the starting point of the second layer is not fixed and the two layers maintain an ordering, which in some ways leads to a simpler description (for instance, the measure has a Gibbs or spatial Markov property, though we  make use of this only for $\c=\d=q=0$). It is not clear to us whether using the two-layer representation for open ASEP under the Liggett's condition in \cite[Theorem 1.1]{bryc2024stationary} (in which the starting points are fixed at zero), one can obtain the LDP in the fan region in a similar way as \cite{bryc2024two}; one issue is that it is not clear if the Radon-Nikodym derivative of the measure with respect to independent random walks can be expressed as  a continuous functional on the path space.

\subsection{Method of proof}\label{sec:method of proof}

As mentioned, our approach relies on a comparison to open TASEP. Fix $\a, \b\geq 0$ and consider the open ASEP system with parameters $\a,\b,\c,\d$ and $q$, and the open TASEP system with parameters $\a$ and $\b$ (i.e., $\c=\d=q=0$). With the two-layer representation introduced in Section~\ref{sec:proof of main thm}, one has an explicit formula for the Radon-Nikodym derivative of the open ASEP two-layer measure with respect to the open TASEP two-layer measure. Suppressing $\a,\b,\c,\d$ in the notation, let $\PP^{(q)}$ and $\EE^{(q)}$ be the probability measure and expectation associated to the open ASEP system mentioned above and $\PP^{(0)}$ and $\EE^{(0)}$ be the same for the open TASEP system (so $\c=\d=0$ for the latter, but not necessarily for the former). Then we may write the just mentioned Radon-Nikodym derivative as $B^{(q)}/\EE^{(0)}[B^{(q)}]$ for a random variable $B^{(q)} = B^{(q,\c,\d)}\in[0,1]$ (often called a Boltzmann factor). Our argument relies on the observation that, if for a given event $A$ we can find an event $E$ such that $\PP^{(0)}(E\mid A) \geq 1-\varepsilon$ and $B^{(q)}\geq 1-\varepsilon$ on $E$, then
\begin{align*}
\PP^{(q)}(A) = \frac{\EE^{(0)}[B^{(q)}\one_A]}{\EE^{(0)}[B^{(q)}]} \geq \PP^{(0)}(A)\cdot\EE^{(0)}[B^{(q)}\mid A] \geq \PP^{(0)}(A)\cdot \EE^{(0)}[B^{(q)}\one_E\mid A] \geq (1-\varepsilon)^2\cdot\PP^{(0)}(A).
\end{align*}
A similar upper bound of $\PP^{(q)}(A)\leq (1-\varepsilon)^{-2}\cdot\PP^{(0)}(A)$ is more simply obtained since one also has $\EE^{(0)}[B^{(q)}]\geq (1-\varepsilon)^2$ given the existence of an event $E$ as above for the whole space.

For our setting, the event $E$ is the event that the two curves remain sufficiently separated for most of their lifetimes. That the Boltzmann factor is essentially 1 when the curves are sufficiently separated has been observed before in the context of ASEP systems, for instance, in the work \cite{aggarwal2024scaling}. Since we are concerned only with probabilities on the large deviation scale, we also do not need a probability lower bound of $1-\varepsilon$ on $\PP^{(0)}(E\mid A)$, but merely something which decays much slower than exponentially in $N$.  Thus the main task in the proof is to establish that $\PP^{(0)}(E\mid A)$ is not too small for events $A$ that arise in the LDP principle (Proposition~\ref{prop:key estimate}). Notice that this is a statement purely about the base measure $\PP^{(0)}$, which is the law of two random walks conditioned to not intersect. 

To prove such an estimate, we make use of Gibbs resampling arguments and monotone coupling properties of the non-intersecting random walk measure. Roughly speaking, we use that the walks ``repel'' each other (Lemma~\ref{lem:two paths monotone}), along with standard fluctuation bounds for unconditioned random walks, to obtain that, with not-too-small probability, the two random walks maintain a separation of at least $N^{3/5}$ for all but $\varepsilon$ proportion of their lifetime. Here, the $3/5$ exponent has no significance, and one could also obtain smaller separation guarantees with higher, e.g., constant order, probability. This argument yields an LDP for the height function evaluated at finitely many points (such choices of events for $A$ are easier to apply the monotonicity arguments for), which can then be upgraded to an LDP on the path space by standard LDP techniques.

\subsection*{Outline of the rest of the paper}
In Section \ref{sec:proof of main thm} we will introduce the two-layer representation of the stationary measure of open ASEP and state a key estimate (Proposition \ref{prop:key estimate}) on random walk measures. In Section \ref{sec:proof of estimate of RW} we will provide the proof of Proposition \ref{prop:key estimate}, and Section~\ref{sec:proof of LDP} contains the proof of the main result, Theorem~\ref{thm:LDP main thm}.

\subsection*{Acknowledgements} We thank Ivan Corwin for many helpful comments, which helped improve the presentation of the paper. We also thank Wlodek Bryc and Dominik Schmid for many valuable discussions. M.H.~was partially supported by NSF grants DMS-1937254 and DMS-2348156. Z.Y. was partially supported by Ivan Corwin’s NSF grant DMS-1811143 as well as the Fernholz Foundation’s ``Summer Minerva Fellows'' program.

\section{The two-layer representation in the fan region}\label{sec:proof of main thm}
In this section we introduce the two-layer representation for the open ASEP stationary measure in the fan region. In Section~\ref{sec:matrix product} we collect some facts about the matrix product ansatz, and in Section~\ref{sec:two layer} we use it to give the two-layer representation in Theorem~\ref{thm:two layer representation}.  In Section~\ref{sec:introduce important estimate}, we state our key estimate of the random walk measures, Proposition \ref{prop:key estimate}. The proof of the latter will be the focus of Section~\ref{sec:proof of estimate of RW}.

\subsection{The matrix product ansatz and Enaud--Derrida representation}\label{sec:matrix product}
In this subsection, we will first introduce the matrix product ansatz for the open ASEP from \cite{derrida1993exact}. We will then introduce the Enaud--Derrida representation of the matrix ansatz, constructed in \cite{enaud2004large}. 
\begin{lemma}[c.f. \cite{derrida1993exact}]\label{lem:MPA} Assume \eqref{eq:conditions open ASEP} holds. Also, assume that there exist matrices $\D$, $\E$, a row vector $\ll W|$, and a column vector $|V\rr$, all of the same (possibly infinite) dimension, and that all admissible finite products of these matrices and vectors are well-defined (i.e., convergent if they are infinite-dimensional) and satisfy the property of associativity. Further assume that
\be\label{eq:DEHP algebra} 
        \D\E-q\E\D=\D+\E, \quad
        \ll W|(\alpha\E-\gamma\D)=\ll W|, \quad
        (\beta\D-\delta\E)|V\rr=|V\rr 
 \ee
    (which is commonly referred to as the DEHP algebra). 
    Then we have 
    \be \label{eq:MPA open ASEP}
\mu_N\lb\tau_1,\dots,\tau_N\rb=\frac{\displaystyle\Bigg{\ll} W\Bigg{|}\prod_{i=1}^N\lb\one_{\tau_i=0}\E+\one_{\tau_i=1}\D\rb\Bigg{|}V\Bigg{\rr}}{\displaystyle \ll W |(\E+\D)^N|V \rr},
\ee 
for any $(\tau_1,\dots,\tau_N)\in\{0,1\}^N$, assuming that the denominator $\ll W|(\E+\D)^N|V\rr$ is nonzero.
\end{lemma}

\begin{remark} \label{rmk:denomenator 0} 
As noted by \cite{mallick1997finite,essler1996representations}, assuming $\ll W|V\rr\neq0$ and finite,  the denominator $\ll W|(\E+\D)^N|V\rr$ of the matrix product ansatz \eqref{eq:MPA open ASEP} being nonzero can be guaranteed by $\a\b\c\d\notin\{q^{-\ell}: \ell\in\ZZ_{\geq0}\}$, where we recall that $\a$, $\b$, $\c$ and $\d$ are given by \eqref{eq:defining ABCD}. Observe that this is always the case in the fan region $\a\b<1$ since $\c,\d\in(-1,0]$.  

As shown in \cite[Appendix A]{mallick1997finite}, in the case above, for any $n_1,m_1,\dots,n_k,m_k\in\ZZ_+$, the ratio
\[
\ll W|\D^{n_1}\E^{m_1}\dots\D^{n_k}\E^{m_k}|V\rr / \ll W|V\rr
\]  
lies in $\RR_+$ and only depends on the parameters $q,\alpha,\beta,\gamma,\delta$ and the choice of $n_i,m_i$. In particular, this ratio does not depend on the specific examples of matrices $\D$, $\E$ and vectors $\ll W|$, $|V\rr$ that satisfy the DEHP algebra \eqref{eq:DEHP algebra}.  \end{remark} 

We will use the following example of $\D$, $\E$, $\ll W|$ and $|V\rr$ constructed by C.\ Enaud and B.\ Derrida  in \cite{enaud2004large}, which is known in the literature as the Enaud--Derrida representation.

\begin{lemma}[c.f. \cite{enaud2004large}]\label{lem:Enaud Derrida representation}
    Assume \eqref{eq:conditions qABCD} and $\a\b<1$. Then the infinite matrices 
    \[
    \D=\frac{1}{1-q}\begin{bmatrix}
  1+\d & 1-q & 0 & 0 &\dots\\
  0 & 1+q \d & 1-q^2&0 &\dots\\
  0& 0 & 1+q^2 \d &  1-q^3   \\
   \vdots & & \ddots&\ddots \\
\end{bmatrix},\quad
   \E=\frac{1}{1-q}\begin{bmatrix}
  1+\c & 0 & 0 & \dots& \\
  1-\c \d & 1+q \c & 0&\dots \\
  0& 1-q \c \d & 1+q^2 \c &  &\\
   \vdots & & \ddots & \ddots\\
\end{bmatrix}, 
\]
and infinite vectors:
\[
\ll W | =[1,\a,\a^2,\a^3,\dots], \quad\quad\quad
|V\rr= \left[1,\b\frac{(\c\d;q)_{1}}{(q;q)_{1}}, \b^2 \frac{(\c\d;q)_{2}}{(q;q)_{2}}, \b^3 \frac{(\c\d;q)_{3}}{(q;q)_{3}},\dots\right]^T
\]
satisfy the DEHP algebra \eqref{eq:DEHP algebra}.
Here and below, for $z\in\mathbb{C}$ and $n\in\mathbb{Z}_{\geq0}\cup\{\infty\}$, we use the $q$-Pochhammer symbol:
$(z;q)_n:=\prod_{j=0}^{n-1}\,(1-z q^j)$. 
Moreover, we have
$\ll W|V\rr=(\a\b\c\d;q)_{\infty}/(\a\b;q)_{\infty}$
\end{lemma}
\begin{proof}
    The first statement appears in \cite[Section 3.2]{enaud2004large}, where our notations $(\a,\b,\c,\d)$ correspond to
    \[
    \lb \frac{1-\rho_a}{\rho_a}, \frac{\rho_b}{1-\rho_b}, -e,-d \rb
    \]
    therein. The second statement follows from the $q$-binomial theorem, see \cite[Equation (1.3.2)]{gasper2011basic}.
\end{proof}

\subsection{The two-layer representation}\label{sec:two layer}
In the fan region $\a\b<1$, the stationary measures of open ASEP can be expressed in terms of certain two-layer measures, which originate from the Enaud--Derrida representation of the matrix product ansatz. In this subsection we will introduce this two-layer representation.

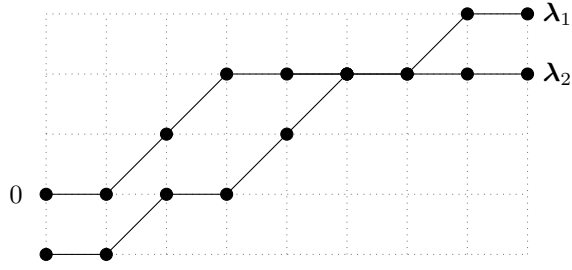
\begin{figure}[ht]
\begin{center}
\begin{tikzpicture}[scale=0.8]
\begin{scope}
    \def\r{0.1}
	\draw[dotted, gray] (0,0) grid (8,4);
        \draw[fill = black] (0,1) circle [radius=\r];
        \draw[fill = black] (1,1) circle [radius=\r];
        \draw[fill = black] (2,2) circle [radius=\r];
        \draw[fill = black] (3,3) circle [radius=\r];
        \draw[fill = black] (4,3) circle [radius=\r];
        \draw[fill = black] (5,3) circle [radius=\r];
        \draw[fill = black] (6,3) circle [radius=\r];
        \draw[fill = black] (7,4) circle [radius=\r];
        \draw[fill = black] (8,4) circle [radius=\r];
        \draw[-][black] (0,1) -- (1,1);
        \draw[-][black] (1,1) -- (2,2);
        \draw[-][black] (2,2) -- (3,3);
        \draw[-][black] (3,3) -- (4,3);
        \draw[-][black] (4,3) -- (5,3);
        \draw[-][black] (5,3) -- (6,3);
        \draw[-][black] (6,3) -- (7,4);
        \draw[-][black] (7,4) -- (8,4);
        \draw[fill = black] (0,0) circle [radius=\r];
        \draw[fill = black] (1,0) circle [radius=\r];
        \draw[fill = black] (2,1) circle [radius=\r];
        \draw[fill = black] (3,1) circle [radius=\r];
        \draw[fill = black] (5,3) circle [radius=\r];
        \draw[fill = black] (4,2) circle [radius=\r];
        \draw[fill = black] (5,3) circle [radius=\r];
        \draw[fill = black] (6,3) circle [radius=\r];
        \draw[fill = black] (7,3) circle [radius=\r];
        \draw[fill = black] (8,3) circle [radius=\r];
        \draw[-][black] (0,0) -- (1,0);
        \draw[-][black] (1,0) -- (2,1);
        \draw[-][black] (2,1) -- (3,1);
        \draw[-][black] (3,1) -- (5,3);
        \draw[-][black] (4,3) -- (5,3);
        \draw[-][black] (5,3) -- (6,3);
        \draw[-][black] (6,3) -- (7,3);
        \draw[-][black] (7,3) -- (8,3); 
        \draw (8.5, 4) node{$\bl_1$};
        \draw (8.5, 3) node{$\bl_2$};
        \draw (-0.5, 1) node{$0$};
\end{scope}
\end{tikzpicture}
\end{center}
\caption{The figure depicts the two layers $\bl_1$ and $\bl_2$, satisfying $\bl=(\bl_1,\bl_2)\in\mathcal{TL}_N$.  }
\label{fig:two layers}
\end{figure} 

\begin{definition}\label{def:two layer Gibbs measure} 
We define the two-layer measures.
\begin{enumerate}
    \item [(1)]
We will use $\mathcal{TL}_N$ to denote the space of 
\[\bl=(\bl_1,\bl_2)\in\mathbb{Z}^{2N+2},\quad \bl_i=(\lambda_i(0),\dots,\lambda_i(N)),\mbox{ }i=1,2\]
satisfying the following conditions:
\begin{enumerate}
    \item[$\bullet$] $\lambda_1(0)=0$,
    \item[$\bullet$] $\lambda_1(j)\geq\lambda_2(j)$ for $j\in\llbracket0,N\rrbracket$,
    \item[$\bullet$] $\lambda_i(j)-\lambda_i(j-1)\in\{0,1\}$ for $i\in\{1,2\}$ and $j\in\llbracket1,N\rrbracket$.
\end{enumerate}
See Figure \ref{fig:two layers} for an illustration.

\item [(2)]
Let $\a,\b,\c,\d,q$ satisfy \eqref{eq:conditions qABCD} and $\a\b<1$. For $\bl\in\mathcal{TL}_N$, we define the two-layer weight
\begin{multline}\label{eq:defining two layer weight}
\wt_N^{(q,\a,\b,\c,\d)}(\bl):=\a^{\lambda_1(0)-\lambda_2(0)}\b^{\lambda_1(N)-\lambda_2(N)}\frac{(\c\d;q)_{\lambda_1(N)-\lambda_2(N)}}{(q;q)_{\lambda_1(N)-\lambda_2(N)}}\times\\
\prod_{j=1}^N\lb\sum_{\mathsf{u},\mathsf{v}\in\{0,1\}}\one_{\substack{\la_1(j)-\la_1(j-1)=\mathsf{u} \\ \la_2(j)-\la_2(j-1)=\mathsf{v} }} W^{(q,\c,\d)}\lb\la_1(j)-\la_2(j)\mid \mathsf{u},\mathsf{v}\rb\rb,
\end{multline}
where for $\mathsf{u},\mathsf{v}\in\{0,1\}$ and $x\in\ZZ_{\geq0}$, we write (with the convention that $0^0 = 1$) 
\be\label{eq:W}
W^{(q,\c,\d)}(x \mid \mathsf{u},\mathsf{v}):=\begin{cases}
    \displaystyle 1-q^x\quad&\mbox{ if }(\mathsf{u},\mathsf{v})=(1,0),\\
    \displaystyle 1+\d q^x\quad&\mbox{ if }(\mathsf{u},\mathsf{v})=(1,1),\\
    \displaystyle 1+\c q^x\quad&\mbox{ if }(\mathsf{u},\mathsf{v})=(0,0),\\
    \displaystyle 1-\c\d q^x\quad&\mbox{ if }(\mathsf{u},\mathsf{v})=(0,1),\\
\end{cases}
\ee 
We then consider the normalization constant
\be\label{eq:def of normalizing constant}
Z_N^{(q,\a,\b,\c,\d)}:=\sum_{\bl\in\mathcal{TL}_N}\wt_N^{(q,\a,\b,\c,\d)}(\bl),
\ee 
which, as shown in Theorem \ref{thm:two layer representation} below, is a finite constant. 
The two-layer measure is a probability measure on the space $\mathcal{TL}_N$ defined, for any $\bl\in\mathcal{TL}_N$, by
\be 
\PP_N^{(q,\a,\b,\c,\d)}(\bl)=\wt_N^{(q,\a,\b,\c,\d)}(\bl)/Z_N^{(q,\a,\b,\c,\d)}.
\ee 
\end{enumerate}
\end{definition}

The next theorem shows that the open ASEP stationary measure is a marginal of the two-layer measure. 
\begin{theorem} \label{thm:two layer representation}
    Let $\a,\b,\c,\d,q$ satisfy \eqref{eq:conditions qABCD} and $\a\b<1$. Then the normalization constant \eqref{eq:def of normalizing constant} is finite, hence
    the two-layer measure $\PP_N^{(q,\a,\b,\c,\d)}$ is a well-defined probability measure on $\mathcal{TL}_N$. The marginal distribution of $\PP_N^{(q,\a,\b,\c,\d)}$ on the first layer $\bl_1$ has the same law as the height function $\tau_1+\dots+\tau_j$, $j\in\llbracket0,N\rrbracket$ of the stationary measure $\mu_N^{(q,\a,\b,\c,\d)}$ of open ASEP.
\end{theorem}
\begin{remark}
    The two-layer representation above is equivalent to the Motzkin path representation given in Section 3 of the first version of \cite{bryc2024stationary} on arXiv. A similar representation in the $\gamma=\delta=0$ case is provided in \cite{brak2006combinatorial} (see also \cite{nestoridi2024approximating}).
Additionally, under Liggett's condition (i.e., $\c = \d = -q$ in our notation), our two-layer representation coincides with the one given in \cite{barraquand2023stationary} (see also \cite{derrida2004asymmetric} for open TASEP case) and with the ``two-layer Gibbs measure'' in \cite{bryc2024stationary} (see also \cite{bryc2024two} for open TASEP).  
\end{remark}
\begin{proof}[Proof of Theorem~\ref{thm:two layer representation}]
    We consider the matrices $\D,\E$ and vectors $\ll W|,|V\rr$ given by the Enaud--Derrida representation \cite{enaud2004large} (see Lemma \ref{lem:Enaud Derrida representation} above). 
    For $i\in\ZZ_{\geq0}$, we use $W_i$ (resp.\ $V_i$) to denote the $i$-th entry of the row vector $\ll W|$ (resp.\ the column vector $|V\rr$).  
    For $i,j\in\ZZ_{\geq0}$, we use $\D_i^j$ (resp.\ $\E_{i}^j$) to denote the $(i,j)$-th entry of the matrix $\D$ (resp.\ the matrix $\E$). In view of \eqref{eq:defining two layer weight}, for any $\bl\in\mathcal{TL}_N$, we have
    \begin{align*}
    \MoveEqLeft[1]
    \wt_N^{(q,\a,\b,\c,\d)}(\bl)\\
    &=(1-q)^N\cdot W_{\la_1(0)-\la_2(0)}\cdot\prod_{j=1}^N
    \lb\one_{\la_1(j)-\la_1(j-1)=0}\E+\one_{\la_1(j)-\la_1(j-1)=1}\D\rb_{\la_1(j-1)-\la_2(j-1)}^{\la_1(j)-\la_2(j)}
    \cdot V_{\la_1(N)-\la_2(N)}.
    \end{align*}
    For any fixed $\bl_1$, we sum over $\bl_2$ and get
    \be \label{eq:marginal on la1}
    \sum_{\bl_2: (\bl_1,\bl_2)\in\mathcal{TL}_N}\wt_N^{(q,\a,\b,\c,\d)}(\bl)=(1-q)^N\Bigg{\ll} W\Bigg{|}\prod_{i=1}^N\lb\one_{\la_1(j)-\la_1(j-1)=0}\E+\one_{\la_1(j)-\la_1(j-1)=1}\D\rb\Bigg{|}V\Bigg{\rr}.
    \ee 
    We next sum over $\bl_1$ and get
    \be \label{eq:marginal la1la2}
    Z_N^{(q,\a,\b,\c,\d)}=\sum_{\bl\in\mathcal{TL}_N}\wt_N^{(q,\a,\b,\c,\d)}(\bl)=(1-q)^N\ll W|(\D+\E)^{N}|V\rr.
    \ee 
    In view of Remark \ref{rmk:denomenator 0} and the fact that $\ll W|V\rr$ is finite, the normalization constant $Z_N^{(q,\a,\b,\c,\d)}$ is finite. The result then follows from combining \eqref{eq:marginal on la1} and \eqref{eq:marginal la1la2}, in light of Lemma \ref{lem:MPA}. This concludes the proof.
\end{proof}

We next provide an asymptotic result for the normalization constant, which will be useful later. It is essentially a consequence of results of \cite{uchiyama2004asymmetric}.
\begin{lemma}\label{lem:log asy of normalization constant}
Let $\a,\b,\c,\d,q$ satisfying \eqref{eq:conditions qABCD} and $\a\b<1$. Then,
    \begin{equation}\label{eq:log asy of normalization constant}
        \lim_{N\rightarrow\infty}\frac{1}{N}\log Z_N^{(q,\a,\b,\c,\d)}=-\log(J(\a,\b)),
    \end{equation}
    where the function $J(\a,\b)$ is defined by \eqref{eq:J}.
\end{lemma}
\begin{proof}
In view of Remark \ref{rmk:denomenator 0}, the ratio
\be\label{eq:gergerg}
\ll W|(\D+\E)^{N}|V\rr/\ll W|V\rr
\ee
does not depend on the specific examples of matrices $\D$, $\E$ and vectors $\ll W|$, $|V\rr$ that satisfy the DEHP algebra \eqref{eq:DEHP algebra}. Using a different example constructed by M.\ Uchiyama, T.\ Sasamoto and M.\ Wadati in \cite{uchiyama2004asymmetric}, known as the USW representation, \cite[Section 6.1]{uchiyama2004asymmetric} computed the asymptotics of \eqref{eq:gergerg}, which, combined with \eqref{eq:marginal la1la2}, yields this result. 
\end{proof}

For another computation of the asymptotics of \eqref{eq:gergerg} using again the USW representation, we refer the reader to \cite{bryc2017asymmetric} and \cite[Lemmas 3.1 and 4.5]{bryc2019limit}.

\subsection{An estimate for the open TASEP two-layer measure} \label{sec:introduce important estimate}
In this subsection, we provide a key estimate for a special type of two-layer measure, $\PP_N^{(0,\a,\b,0,0)}$, which corresponds to the open TASEP. These measures can be viewed as endpoint reweightings of two random walks conditioned not to intersect. Using this estimate, we will compare the two-layer weights for the open ASEP with those for the open TASEP.
\begin{definition}[Height profile]\label{def:height profile of two layer}
    For a two-layer configuration $\bl\in\mathcal{TL}_N$, we denote by $h_N[\bl]\in C_0([0,1],\RR)$ the (re-scaled) height profile corresponding to it, defined as
\[h_N[\bl]\lb\frac{k}{N}\rb=\frac{\lambda_1(k)}{N},\quad k=0,\dots,N.\]
and by linear interpolation.  
\end{definition}
We will use the following key estimate of the two-layer measure $\PP_N^{(0,\a,\b,0,0)}$, Proposition~\ref{prop:key estimate}, in the proof of the large deviation principle. Roughly, it states that, with not too low probability, $\bl_1$ and $\bl_2$ are separated on most of $\llbracket0,N\rrbracket$. It will allow us to compare the weight function under the case of general parameters in the fan region and the $q=0$ case, as will be captured in Corollary~\ref{cor:key estimate}. The proof of Proposition \ref{prop:key estimate} will be the focus of Section \ref{sec:proof of estimate of RW} and will appear in Section \ref{subsec:proof of estimate of RW}. 
Before stating it, observe from Definition \ref{def:two layer Gibbs measure} that for $\bl\in\mathcal{TL}_N$, we have
\[
\wt_N^{(0,\a,\b,0,0)}(\bl)=\a^{\lambda_1(0)-\lambda_2(0)}\b^{\lambda_1(N)-\lambda_2(N)},\quad
\PP_N^{(0,\a,\b,0,0)}(\bl)=\wt_N^{(0,\a,\b,0,0)}(\bl)/Z_N^{(0,\a,\b,0,0)}.
\]
\begin{proposition}\label{prop:key estimate}
    Let $\a,\b\geq0$ satisfy $\a\b<1$. Fix $d\in\mathbb{N}$ and $0<\th_1<\dots<\th_d<\th_{d+1}=1$ and $0<u_j<v_j$, $j\in\llb1,d+1\rrb$ satisfying 
    $v_{d+1}-u_j<\th_{d+1}-\th_j$ for $j\in\llb1,d\rrb$. Then for any fixed $r\in\mathbb{N}$ and $\ep>0$ there exists $N_0\in\mathbb{N}$ such that for $N\geq N_0$,
    \be
    \PP_N^{(0,\a,\b,0,0)}\lb  \#\{j: \lambda_1(j)-\lambda_2(j)\leq r\}\leq \lceil\varepsilon N\rceil \mbox{ } \Big{|} \mbox{ } h_N[\bl](\th_j)\in(u_j,v_j), j\in\llb1,d+1\rrb\rb\geq e^{-N^{4/5}}.
    \ee 
\end{proposition}

As a corollary, we next give a comparison result between the two-layer weights for open ASEP and the ones for open TASEP.

\begin{corollary}\label{cor:key estimate}
    Let $\a,\b,\c,\d,q$ satisfying \eqref{eq:conditions qABCD} and $\a\b<1$. Fix $d\in\mathbb{N}$ and $0<\th_1<\dots<\th_d<\th_{d+1}=1$ and $0<u_j<v_j$, $j\in\llb1,d+1\rrb$ satisfying 
    $v_{d+1}-u_j<\th_{d+1}-\th_j$ for $j\in\llb1,d\rrb$. Then we have
\be  \label{eq:log asy only need to prove}
\lim_{N\to\infty}\frac{1}{N} \log
\frac{\sum_{\bl\in\mathcal{TL}_N: h_N[\bl](\th_j)\in(u_j,v_j), j\in\llb1,d+1\rrb}\wt_N^{(q,\a,\b,\c,\d)}(\bl)}{\sum_{\bl\in\mathcal{TL}_N: h_N[\bl](\th_j)\in(u_j,v_j), j\in\llb1,d+1\rrb}\wt_N^{(0,\a,\b,0,0)}(\bl)}=0.
\ee
\end{corollary}
\begin{proof}[Proof of Corollary \ref{cor:key estimate} assuming Proposition \ref{prop:key estimate}]
For simplicity, we will write (where $\mathrm{Val}$ is short for ``values'')
\[
\mathrm{Val}:=\{\bl\in\mathcal{TL}_N: h_N[\bl](\th_j)\in(u_j,v_j), j\in\llb1,d+1\rrb\}.
\]
By Definition \ref{def:two layer Gibbs measure}, for any  $\bl\in\mathcal{TL}_N$, we have 
\begin{multline}\label{eq:egrgerergarf}
\wt_N^{(q,\a,\b,\c,\d)}(\bl)=\wt_N^{(0,\a,\b,0,0)}(\bl)\frac{(\c\d;q)_{\lambda_1(N)-\lambda_2(N)}}{(q;q)_{\lambda_1(N)-\lambda_2(N)}}\times\\
\prod_{j=1}^N\lb\sum_{\mathsf{u},\mathsf{v}\in\{0,1\}}\one_{\substack{\la_1(j)-\la_1(j-1)=\mathsf{u} \\ \la_2(j)-\la_2(j-1)=\mathsf{v} }} W^{(q,\c,\d)}\lb\la_1(j)-\la_2(j)\mid\mathsf{u},\mathsf{v}\rb\rb.
\end{multline}
Recall that $\a,\b\geq0$, $\c,\d\in(-1,0]$ and $q\in[0,1)$.
Observe from \eqref{eq:W} that, 
\be \label{eq:first bound W}
W^{(q,\c,\d)}\lb\la_1(j)-\la_2(j)\mid\mathsf{u},\mathsf{v}\rb\leq1
\ee 
and, as long as $(\lambda_1(j)-\lambda_2(j), \mathsf u, \mathsf v) \neq (0,1,0)$,
\be \label{eq:second bound W}
W^{(q,\c,\d)}\lb\la_1(j)-\la_2(j)\mid\mathsf{u},\mathsf{v}\rb\geq \max\lb \mathsf{g}, 1-q^{\la_1(j)-\la_2(j)}\rb,\quad \mathsf{g}:=\min(1+\c,1+\d,1-\c\d)\in(0,1].
\ee 
In particular, using \eqref{eq:first bound W}, the second line of \eqref{eq:egrgerergarf} lies in $(0,1]$. 
One can also observe that 
\be \label{eq:bound on cd;q}
1-\c\d\leq (\c\d;q)_n/(q;q)_n\leq 1/(q;q)_{\infty}\quad\mbox{ for all }n\in\ZZ_{\geq0}.
\ee 
Using \eqref{eq:egrgerergarf} we then have, for all $\bl\in\mathcal{TL}_N$,
\be\label{eq:upper bound weight}
\wt_N^{(q,\a,\b,\c,\d)}(\bl)\leq \wt_N^{(0,\a,\b,0,0)}(\bl)/(q;q)_{\infty}.
\ee
Therefore
\be \label{eq:limsupewfwevf}
\limsup_{N\to\infty}\frac{1}{N} \log
\frac{\sum_{\bl\in\mathrm{Val}}\wt_N^{(q,\a,\b,\c,\d)}(\bl)}{\sum_{\bl\in\mathrm{Val}}\wt_N^{(0,\a,\b,0,0)}(\bl)}\leq0.
\ee 

We next show the inequality in the other direction.  Fix any $r\in\mathbb{N}$ and $\varepsilon>0$. We will write (where Sep is short for ``separation'')
\[
\mathrm{Sep}:=\{\bl\in\mathcal{TL}_N: \#\{j: \lambda_1(j)-\lambda_2(j)\leq r\}\leq \lceil\varepsilon N\rceil\}.
\]
Next we observe that, since $\bl\in\mathcal T$, for no $j=1,\ldots,N$ in \eqref{eq:egrgerergarf} will it be the case that $(\lambda_1(j)-\lambda_2(j), \mathsf u, \mathsf v) = (0,1,0)$, $\lambda_1(j) - \lambda_1(j-1) = \mathsf{u}$, and $\lambda_2(j) - \lambda_2(j-1) = \mathsf{v}$; otherwise, it would hold that $\lambda_1(j-1) < \lambda_2(j-1)$. So we may apply the bound \eqref{eq:second bound W} to each factor in \eqref{eq:egrgerergarf}. For each $\bl\in\mathrm{Sep}$,
by combining  \eqref{eq:second bound W} and \eqref{eq:bound on cd;q},
in view of \eqref{eq:egrgerergarf}, we have
\[
\wt_N^{(q,\a,\b,\c,\d)}(\bl)\geq(1-\c\d)(1-q^{r+1})^{N-\lceil\varepsilon N\rceil}\mathsf{g}^{\lceil\varepsilon N\rceil }\wt_N^{(0,\a,\b,0,0)}(\bl).
\]
Therefore,
\begin{equation}\label{eq:vreagv}
    \sum_{\bl\in\mathrm{Val}}\wt_N^{(q,\a,\b,\c,\d)}(\bl)\geq\sum_{\bl\in\mathrm{Val}\cap\mathrm{Sep}}\wt_N^{(q,\a,\b,\c,\d)}(\bl)\geq
    (1-\c\d)(1-q^{r+1})^{N-\lceil\varepsilon N\rceil}\mathsf{g}^{\lceil\varepsilon N\rceil } \sum_{\bl\in\mathrm{Val}\cap\mathrm{Sep}}\wt_N^{(0,\a,\b,0,0)}(\bl). 
\end{equation}
Using Proposition \ref{prop:key estimate},  there exists $N_0\in\mathbb{N}$ such that for $N\geq N_0$, 
\be \label{eq:vger}
\sum_{\bl\in\mathrm{Val}\cap\mathrm{Sep}}\wt_N^{(0,\a,\b,0,0)}(\bl)\geq e^{-N^{4/5}} \sum_{\bl\in\mathrm{Val} }\wt_N^{(0,\a,\b,0,0)}(\bl).
\ee 
Combining \eqref{eq:vreagv} with \eqref{eq:vger}, we have that for $N\geq N_0$,
\[
\frac{\sum_{\bl\in\mathrm{Val}}\wt_N^{(q,\a,\b,\c,\d)}(\bl)}{\sum_{\bl\in\mathrm{Val}}\wt_N^{(0,\a,\b,0,0)}(\bl)} \geq (1-\c\d) (1-q^{r+1})^{N-\lceil\varepsilon N\rceil}\mathsf{g}^{\lceil\varepsilon N\rceil  }  e^{-N^{4/5}}.
\] 
Therefore
\begin{equation*}
    \begin{split}
        &\liminf_{N\to\infty}\frac{1}{N} \log
\frac{\sum_{\bl\in\mathrm{Val}}\wt_N^{(q,\a,\b,\c,\d)}(\bl)}{\sum_{\bl\in\mathrm{Val}}\wt_N^{(0,\a,\b,0,0)}(\bl)}\\&\geq\liminf_{N\to\infty}\lb\frac{1}{N} \log\lb (1-\c\d)(1-q^{r+1})^{N-\lceil\varepsilon N\rceil}\mathsf{g}^{\lceil\varepsilon N\rceil } e^{-N^{4/5}}\rb \rb=(1-\ep)\log(1-q^{r+1})+\ep\log(\mathsf{g}).
    \end{split}
\end{equation*}
Take $r\rightarrow\infty$, $\varepsilon \rightarrow0+$, the right hand side converges to $0$. Combining with \eqref{eq:limsupewfwevf} we conclude the proof.
\end{proof}

\section{Proof of two-layer separation}
\label{sec:proof of estimate of RW}
In this section, we provide the proof of the key estimate on separation of the two layers in the $q=0$ case (Proposition \ref{prop:key estimate}). The proof, appearing in Section \ref{subsec:proof of estimate of RW}, will rely on the properties of random walks discussed in Section \ref{subsec:useful results of random walks}.

\subsection{Useful results of random walks}
\label{subsec:useful results of random walks}
In this subsection, we provide some results on Bernoulli random walk bridges, which will be useful later in our proof.

\begin{definition}[Bernoulli paths and random walk bridges]
     Let $N\in\mathbb{N}$.
    \begin{enumerate}
        \item [(1)] A \emph{Bernoulli path} $L$ is a function $L:\llb0,N\rrb\rightarrow\ZZ$ such that $L(j)-L(j-1)\in\{0,1\}$ for $j\in\llbracket1,N\rrbracket$.
        \item [(2)] Let $f:\llbracket0,N\rrbracket\rightarrow\ZZ\cup\{\infty\}$, $g:\llbracket0,N\rrbracket\rightarrow\ZZ\cup\{-\infty\}$ be two functions and let $x,y\in\ZZ$. We denote by $\Omega(N,x,y,f,g)$ the set of Bernoulli paths $L$ such that $L(0)=x$, $L(N)=y$ and $L\leq f$, $L\geq g$ over $\llb0,N\rrb$. If $\Omega(N,x,y,f,g)\neq\emptyset$, then we denote by $\PP^{N,x,y,f,g}$ the uniform measure on it. When $f=\infty$ and $g=-\infty$ we also write $\Omega(N,x,y)$ and $\PP^{N,x,y}$ to denote $\Omega(N,x,y,\infty,-\infty)$ and $\PP^{N,x,y,\infty,-\infty}$. We call a path sampled from such a measure a \emph{Bernoulli random walk bridge}. 
        \item [(3)] Let $x_1,x_2,y_1,y_2\in\ZZ$. We denote by $\Omega(N,(x_1,x_2),(y_1,y_2))$ the set of pairs of Bernoulli paths $(L_1,L_2)$ such that $L_1\geq L_2$ on $\llb0,N\rrb$, $L_1(0)=x_1$, $L_2(0)=x_2$, $L_1(N)=y_1$ and $L_2(N)=y_2$. If $\Omega(N,(x_1,x_2),(y_1,y_2))\neq\emptyset$, we denote by $\PP^{N,(x_1,x_2),(y_1,y_2)}$ the uniform measure on it. We call a pair of paths sampled from such a measure \emph{non-intersecting Bernoulli random walk bridges}. 
    \end{enumerate}
\end{definition}

We next record a useful property of non-intersecting random walk bridges, namely their spatial Markov or Gibbs property. It follows immediately from the fact that the law of these processes is uniform on the appropriate set of Bernoulli paths. 

\begin{lemma}\label{lem:gibbs}
    Let $N\in\mathbb{N}$, $f:\intint{0,N}\to\ZZ\cup\{\infty\}$, $g:\intint{0,N}\to\ZZ\cup\{-\infty\}$ and $x,y\in\ZZ$ be such that $\Omega(N,x,y,f,g)\neq\emptyset$. Let $\intint{a,b}\subseteq \intint{0,N}$. Then for any $A\subseteq \ZZ^{\intint{a,b}}$,
    \begin{align*}
        \PP^{N,x,y,f,g}\left(L\in A\mid (L(x):x\not\in\intint{a+1,b-1})\right) = \PP^{\intint{a,b},L(a),L(b),f|_{\intint{a,b}},g|_{\intint{a,b}}}\left(L\in A\right),
    \end{align*}
    almost surely, where, in a slight abuse of notation, $\PP^{\intint{a,b},L(a),L(b),f|_{\intint{a,b}},g|_{\intint{a,b}}}$ is the law of a random walk bridge on $\intint{a,b}$ with boundary values $L(a)$ and $L(b)$ and conditioned to not intersect $f$ and $g$. 
    
    We also have that, for $x_1,x_2,y_1,y_2\in\ZZ$ such that $\Omega(N,(x_1,x_2),(y_1,y_2))\neq\emptyset$ and any $A\subseteq\ZZ^{\intint{0,N}}$, almost surely
    \begin{align*}
        \PP^{N,(x_1,x_2), (y_1,y_2)}(L_1\in A\mid L_2) &= \PP^{N,x_1, y_1, \infty, L_2}(L_1\in A) \quad\text{and}\\
        \PP^{N,(x_1,x_2), (y_1,y_2)}(L_2\in A\mid L_1) &= \PP^{N,x_2, y_2, L_1, -\infty}(L_2\in A).
    \end{align*} 
\end{lemma}

The following lemma allows us to monotonically couple two pairs of non-intersecting Bernoulli random walk bridges with ordered boundary data. It has been proven a number of times, perhaps first in \cite{cohn2000local}, but for notation closer to our own we refer the reader to \cite{dimitrov2021REU} or \cite{serio2023tightness}.

\begin{lemma}[\cite{cohn2000local,dimitrov2021REU,serio2023tightness}] \label{lem:monotonicity}
We have the monotone couplings of Bernoulli paths. Fix $N\in\mathbb{N}$.
\begin{enumerate}
    \item [(1)]    Suppose $g^{\mathrm{b}}, g^{\mathrm{t}}: \llbracket 0, N \rrbracket  \rightarrow \mathbb{Z}\cup\{-\infty\}$ with $g^{\mathrm{b}}\leq g^{\mathrm{t}}$ on $\llbracket0,N\rrbracket$. Suppose $x^{\mathrm{b}},x^{\mathrm{t}},y^{\mathrm{b}},y^{\mathrm{t}}\in\ZZ$ with $x^{\mathrm{b}}\leq x^{\mathrm{t}}$ and $y^{\mathrm{b}}\leq y^{\mathrm{t}}$. Assume that $\Omega(N,x^{\mathrm{b}},y^{\mathrm{b}},\infty,g^{\mathrm{b}})$ and $\Omega(N,x^{\mathrm{t}},y^{\mathrm{t}},\infty,g^{\mathrm{t}})$ are nonempty. Then there exists a coupling between $L^{\mathrm{b}}\sim\PP^{N,x^{\mathrm{b}},y^{\mathrm{b}},\infty,g^{\mathrm{b}}}$ and $L^{\mathrm{t}}\sim\PP^{N,x^{\mathrm{t}},y^{\mathrm{t}},\infty,g^{\mathrm{t}}}$ such that $L^{\mathrm{b}}\leq L^{\mathrm{t}}$ on $\llbracket0,N\rrbracket$.
    \item [(2)]    Suppose $f^{\mathrm{b}}, f^{\mathrm{t}}: \llbracket 0, N \rrbracket  \rightarrow \mathbb{Z}\cup\{\infty\}$ with $f^{\mathrm{b}}\leq f^{\mathrm{t}}$ on $\llbracket0,N\rrbracket$. Suppose $x^{\mathrm{b}},x^{\mathrm{t}},y^{\mathrm{b}},y^{\mathrm{t}}\in\ZZ$ with $x^{\mathrm{b}}\leq x^{\mathrm{t}}$ and $y^{\mathrm{b}}\leq y^{\mathrm{t}}$. Assume that $\Omega(N,x^{\mathrm{b}},y^{\mathrm{b}},f^{\mathrm{b}},-\infty)$ and $\Omega(N,x^{\mathrm{t}},y^{\mathrm{t}},f^{\mathrm{t}},-\infty)$ are nonempty. Then there exists a coupling between $L^{\mathrm{b}}\sim\PP^{N,x^{\mathrm{b}},y^{\mathrm{b}},f^{\mathrm{b}},-\infty}$ and $L^{\mathrm{t}}\sim\PP^{N,x^{\mathrm{t}},y^{\mathrm{t}},f^{\mathrm{t}},-\infty}$ such that $L^{\mathrm{b}}\leq L^{\mathrm{t}}$ on $\llbracket0,N\rrbracket$.
\end{enumerate}
\end{lemma}

Using Lemma~\ref{lem:monotonicity} we obtain a useful correlation inequality.

\begin{lemma}\label{lem:two paths monotone}
    Fix $N\in\mathbb{N}$. Let $x_1,x_2,y_1,y_2\in\ZZ$ such that $\Omega(N,(x_1,x_2),(y_1,y_2))\neq\emptyset$. Then for any functions $h_1,h_2: \llbracket 0, N \rrbracket  \rightarrow \mathbb{Z}$ we have 
    \be 
    \PP^{N,(x_1,x_2),(y_1,y_2)}(L_1\geq h_1,L_2\leq h_2)\geq \PP^{N,x_1,y_1}(L_1\geq h_1)\cdot\PP^{N,x_2,y_2}(L_2\leq h_2).
    \ee 
\end{lemma}
\begin{proof}
We can assume that $\PP^{N,x_1,y_1}(L_1\geq h_1),\PP^{N,x_2,y_2}(L_2\leq h_2)>0$.
    We first show that 
    \be \label{eq:eragvggr}
    \PP^{N,(x_1,x_2),(y_1,y_2)}(L_2\leq h_2)\geq \PP^{N,x_2,y_2}(L_2\leq h_2).
    \ee
    Conditioned on a fixed $L_1\in\Omega(N,x_1,y_1)$, the path $L_2$ has law $\PP^{N,x_2,y_2,L_1,-\infty}$. By the monotone coupling Lemma \ref{lem:monotonicity} (2), we have
    \[
    \PP^{N,x_2,y_2,L_1,-\infty}(L_2\leq h_2) \geq \PP^{N,x_2,y_2}(L_2\leq h_2).
    \]
    This concludes the proof of \eqref{eq:eragvggr}. We next show that
    \be\label{eq:vraegvera}
    \PP^{N,(x_1,x_2),(y_1,y_2)}(L_1\geq h_1\mid L_2\leq h_2)\geq\PP^{N,x_1,y_1}(L_1\geq h_1).
    \ee 
    Condition on a fixed $L_2\in\Omega(N,x_2,y_2)$ satisfying $L_2\leq h_2$, if the path $L_1$ exists, it has law $\PP^{N,x_1,y_1,\infty,L_2}$. By the monotone coupling Lemma \ref{lem:monotonicity} part (1), we have
    \[
    \PP^{N,x_1,y_1,\infty,L_2}(L_1\geq h_1)\geq\PP^{N,x_1,y_1}(L_1\geq h_1).
    \]
    This concludes the proof of \eqref{eq:vraegvera}. Combining with \eqref{eq:eragvggr} we conclude the proof.
\end{proof}

We will also need a second correlation inequality, namely the Fortuin-Kasteleyn-Ginibre (FKG) inequality, saying that non-intersecting random walk bridges are positively associated. For this we introduce the next definition.

\begin{definition}
    Define a partial ordering on the lattice $\ZZ^n$ by $x\preceq y$ if and only if $x_j\leq y_j$ for $j=1,\dots,n$. A subset $B\subset\ZZ^n$ is called increasing (resp. decreasing) if for any $x\preceq y$ and $x\in B$ we have $y\in B$ (resp. for any $x\preceq y$ and $y\in B$ we have $x\in B$). 
\end{definition}

\begin{lemma}[FKG inequality]\label{lem:FKG inequality} 
Fix $N\in\mathbb{N}$. Let $x,y\in\ZZ$, $f:\llb0,N\rrb\to\ZZ\cup\{\infty\}$, and $g:\llb0,N\rrb\to\ZZ\cup\{-\infty\}$ be such that $\Omega(N,x,y, f,g)\neq\emptyset$. Suppose $B,C\subset\ZZ^N$ are either both increasing or both decreasing. Then,
    \[
    \PP^{N,x,y,f,g}(L\in B\cap C)\geq \PP^{N,x,y,f,g}(L\in B)\cdot\PP^{N,x,y,f,g}(L\in C).
    \]
\end{lemma} 

Lemma~\ref{lem:FKG inequality} is a well-known fact, but as we could not find an explicitly quotable statement in the literature, we give a brief proof.

\begin{proof}[Proof of Lemma~\ref{lem:FKG inequality}]
    It is standard (see, for instance \cite[Theorem 4.11]{georgii2001random}) that the FKG inequality is implied by the statement that, for any $z\in\llb0,N\rrb$, $t\in\ZZ$, and $\eta^{\shortuparrow}, \eta^{\shortdownarrow}\in \ZZ^{\llb0,N\rrb\setminus\{z\}}$ such that $\eta^{\shortdownarrow}\preceq \eta^{\shortuparrow}$,
    $$\PP^{N,x,y,f,g}(L(z)\geq t \mid (L(z') : z'\neq z) = \eta^{\shortdownarrow}) \leq \PP^{N,x,y,f,g}(L(z)\geq t \mid (L(z') : z'\neq z) = \eta^{\shortuparrow}),$$
    where we assume that both conditional probabilities are defined, i.e., the probability of both the conditioning events are positive. That the above inequality holds is a special case of Lemma \ref{lem:monotonicity} but is also easy to check directly since conditional on $(L(z'):z\neq z')$, $L(z)$ can take only one or two values, namely $L(z-1)$ and/or $L(z+1)$. 
\end{proof}

The following is a standard supremum bound on random walk bridges.

\begin{lemma}[c.f.\ Lemma 6.13 in \cite{aggarwal2024scaling}]\label{lem:bound on fluctuation of random walk path}
    There exists $C,c>0$ such that for any  $N\in\mathbb{N}$, $m\in\llb0,N\rrb$, and $M>0$, 
    \be 
    \PP^{N,0,m}\lb\sup_{j\in\llb0,N\rrb}\left|L(j)-\frac{mj}{n}\right|\leq M\rb\geq 1-C\exp\lb-\frac{cM^2}{N}\rb.
    \ee 
\end{lemma}

Finally, we state a simple comparison statement for the tails of one-point distributions of random walk bridges.

\begin{lemma}\label{lem:hypergeometric bound}
    Suppose $N\in\mathbb{N}$ and $m,h\in\llb0,N\rrb$. Let $0\leq s\leq t$ satisfying $h-s\leq N-m$. Then we have
    \be 
    \frac{\PP^{N,0,h}(L(m)\leq s)}{\PP^{N,0,h}(L(m)\leq t)}\geq N^{-1-t+s}.
    \ee 
\end{lemma}
\begin{proof}
    Observe that
    \[
    \PP^{N,0,h}(L(m)=k)= {{m}\choose{k}}{ {N-m}\choose{h-k}}\Big{/}{{N}\choose{h}},
    \]
    which is nonzero if and only if $0\leq k\leq m$ and $0\leq h-k\leq N-m$.  
We next show the following statement: there exists $k_s$ and $k_t$ satisfying $|k_t-k_s|\leq t-s$, such that $\PP^{N,0,h}(L(m)=k_s)$ is the maximum of $\PP^{N,0,h}(L(m)=k)$ over all $k\in\llb0,s\rrb$, and a likewise statement holds for $k_t$. 
    Recall that the mass of a hypergeometric distribution first increases and later decreases, and takes its maximum either at one integer or at two consecutive integers. If $\PP^{N,0,h}(L(m)=k)$ takes the maximum over all $k\in\llb0,h\rrb$ at $k=k_{\operatorname{max}}$, then (1) if $s\leq t\leq k_{\operatorname{max}}$ then $k_s=s$ and $k_t=t$; (2) if $s\leq k_{\operatorname{max}}\leq t$ then $k_s=s$ and $k_t=k_{\operatorname{max}}$; and (3) if $k_{\operatorname{max}}\leq s\leq t$ then $k_s=k_t=k_{\operatorname{max}}$. In all these cases we have $|k_t-k_s|\leq t-s$. When $\PP^{N,0,h}(L(m)=k)$ takes the maximum at $k=k_{\operatorname{max}}, k_{\operatorname{max}}+1$, then a similar argument justifies $|k_t-k_s|\leq t-s$.
    Therefore we have
    \[
    \frac{\displaystyle \PP^{N,0,h}(L(m)\leq s)}{\displaystyle \PP^{N,0,h}(L(m)\leq t)}\geq\frac{1}{t}\frac{\displaystyle {{m}\choose{k_s}}{ {N-m}\choose{h-k_s}}}{\displaystyle {{m}\choose{k_t}}{ {N-m}\choose{h-k_t}}}\geq\frac{1}{t}N^{-|k_s-k_t|}\geq N^{-1-t+s},
    \]
    where we have used the fact that, for any $r_1,r_2\in\llb0,n\rrb$, we have
    \[ \displaystyle {n\choose r_1}\Big{/} {n\choose r_2}=\frac{r_2!(n-r_2)!}{r_1!(n-r_1)!}\geq n^{-|r_1-r_2|}.\] 
    We conclude the proof.
\end{proof}

\subsection{Proof of Proposition \ref{prop:key estimate}}\label{subsec:proof of estimate of RW}
In this subsection we will provide the proof of Proposition \ref{prop:key estimate}. For the convenience of the reader, we restate the result as follows. We recall that, for $\bl\in\mathcal{TL}_N$, we have
\[
\wt_N^{(0,\a,\b,0,0)}(\bl)=\a^{\lambda_1(0)-\lambda_2(0)}\b^{\lambda_1(N)-\lambda_2(N)},\quad
\PP_N^{(0,\a,\b,0,0)}(\bl)=\wt_N^{(0,\a,\b,0,0)}(\bl)/Z_N^{(0,\a,\b,0,0)}.
\]

\begin{proposition}\label{prop:key estimate again}
    Let $\a,\b\geq0$ satisfying $\a\b<1$. Fix $d\in\mathbb{N}$, $0<\th_1<\dots<\th_d<\th_{d+1}=1$ and $0<u_j<v_j$, $j\in\llb1,d+1\rrb$ satisfying 
    $v_{d+1}-u_j<1-\th_j$ for $j\in\llb1,d\rrb$. Then for any fixed $r\in\mathbb{N}$ and $\ep>0$ there exists $N_0\in\mathbb{N}$ such that for $N\geq N_0$, we have
    \be\label{eq:key estimate }
    \PP_N^{(0,\a,\b,0,0)}\lb  \#\{j: \lambda_1(j)-\lambda_2(j)\leq r\}\leq \lceil\varepsilon N\rceil  \mbox{ }\Big{|} \mbox{ }h_N[\bl](\th_j)\in(u_j,v_j), j\in\llb1,d+1\rrb\rb\geq e^{-N^{4/5}}.
    \ee 
\end{proposition}

\begin{proof}
    We split the proof into two steps. In the first step we show we can obtain a separation of $N^{3/4}$ between $\lambda_1$ and $\lambda_2$ at $\lfloor\theta_jN\rfloor$ for $j=1,\ldots,d$ with good probability, and in the second step that separation of this order can be maintained at the intermediate locations. \\
	
\noindent\textbf{Step 1.} In this step we will show that there exists $N_1\in\mathbb{N}$ such that for $N\geq N_1$,
\begin{equation}\label{eq:step 1 estimate}
\begin{split}
\MoveEqLeft[36]
    \PP_N^{(0,\a,\b,0,0)}\lb  \lambda_1(\lfloor\th_jN\rfloor)-\lambda_2(\lfloor\th_jN\rfloor)\geq \lceil N^{3/4}\rceil, j\in\llb1,d\rrb\ \Big{|}\  h_N[\bl](\th_j)\in(u_j,v_j), j\in\llb1,d+1\rrb\rb\\
    &\geq N^{-2d \lceil N^{3/4}\rceil}.
\end{split}
\end{equation} 
By our assumptions on $u_j$ and $v_j$, we can choose $N_1\in\mathbb{N}$ such that for $N\geq N_1$, we have 
\be\label{eq:step 1 choose N0}
\lfloor u_jN\rfloor\geq \lceil N^{3/4}\rceil \mbox{ for }j\in\llb1,d+1\rrb\quad\mbox{ and }\quad\lfloor v_{d+1}N\rfloor-\lfloor u_jN\rfloor+\lceil N^{3/4}\rceil\leq N-\lfloor\th_jN\rfloor\mbox{ for }j\in\llb1,d\rrb.
\ee  
Note that $h_N[\bl](\th_j)\in(u_j,v_j)$ implies $\lambda_1(\lfloor\th_jN\rfloor)\in\llb\lfloor u_jN\rfloor,\lfloor v_jN\rfloor\rrb$. Conditioned on the path $\bl_1$ and $\la_2(0)=x$, $\la_2(N)=y$, the path $\bl_2$ (if such a path exists with the mentioned boundary conditions, i.e., $\Omega(N,x,y,\bl_1,-\infty)\neq\emptyset$) has law $\PP^{N,x,y,\bl_1,-\infty}$. Therefore, to show \eqref{eq:step 1 estimate}, we only need to show the following statement: for $N\geq N_1$, a fixed Bernoulli path $\bl_1$ starting from $\la_1(0)=0$ and satisfying $\lambda_1(\lfloor\th_jN\rfloor)\in\llb\lfloor u_jN\rfloor,\lfloor v_jN\rfloor\rrb$ for $j\in\llb1,d+1\rrb$, and fixed $x,y\in\ZZ$ satisfying $\Omega(N,x,y,\bl_1,-\infty)\neq\emptyset$, we have 
\be\label{eq:to show separation first step}
\PP^{N,x,y,\bl_1,-\infty}\lb
L(\lfloor\th_jN\rfloor)\leq \lambda_1(\lfloor\th_jN\rfloor)-\lceil N^{3/4}\rceil,
 j\in\llb1,d\rrb\rb\geq N^{-2d \lceil N^{3/4} \rceil }.
\ee 
We will argue this using the FKG inequality (Lemma~\ref{lem:FKG inequality}) and monotonicity (Lemma~\ref{lem:monotonicity}). For this we will need raised versions of $\bl_1$, denoted
$\widetilde{\bl}_1^{\lfloor\th_jN\rfloor}:\llb0,N\rrb\rightarrow\ZZ$  for $j\in\llb1,\ldots,d\rrb$ and given by
\[
\widetilde{\bl}_1^{\lfloor\th_jN\rfloor}(x)=\begin{cases}
    \la_1\lb\lfloor\th_jN\rfloor\rb &\mbox{ if } x=\lfloor\th_jN\rfloor,\\\
    \infty &\mbox{ otherwise}.
\end{cases}
\]
For such $\bl_1$ and $x,y$ as in \eqref{eq:to show separation first step}, we have
\begin{align*}
\MoveEqLeft[2]
    \PP^{N,x,y,\bl_1,-\infty}\lb L(\lfloor\th_jN\rfloor)\leq \lambda_1(\lfloor\th_jN\rfloor)-\lceil N^{3/4}\rceil, j\in\llb1,d\rrb\rb\\
    &\geq\prod_{j=1}^d\PP^{N,x,y,\bl_1,-\infty}\lb L(\lfloor\th_jN\rfloor)\leq \lambda_1(\lfloor\th_jN\rfloor)-\lceil N^{3/4}\rceil \rb\\
    &\geq\prod_{j=1}^d\PP^{N,0,\la_1(N),\widetilde{\bl}_1^{\lfloor\th_jN\rfloor},-\infty}\lb L(\lfloor\th_jN\rfloor)\leq \lambda_1(\lfloor\th_jN\rfloor)-\lceil N^{3/4}\rceil\rb\\
    &=\prod_{j=1}^d\PP^{N,0,\la_1(N) }\lb  L(\lfloor\th_jN\rfloor)\leq \lambda_1(\lfloor\th_jN\rfloor)- \lceil N^{3/4}\rceil\mbox{ }\Big{|}\mbox{ }
    L(\lfloor\th_jN\rfloor)\leq\lambda_1(\lfloor\th_jN\rfloor)\rb\\
    &\geq\lb N^{-\lceil N^{3/4}\rceil-1}\rb^d\geq N^{-2d \lceil N^{3/4}\rceil}, 
\end{align*}
 where in the first step we use the FKG inequality (Lemma \ref{lem:FKG inequality}), in the second step we use the monotone coupling (Lemma \ref{lem:monotonicity},  raising both the boundary values and the upper boundary curve; recall that $x\leq \lambda_1(0) = 0$ and $y\leq \lambda_1(N)$), and
in the fourth step we use Lemma \ref{lem:hypergeometric bound}, whose conditions are satisfied by \eqref{eq:step 1 choose N0}. We conclude the proof of \eqref{eq:step 1 estimate}.\\
	
\noindent\textbf{Step 2.}
In this step we will show that there exists $N_2\in\mathbb{N}$ such that for $N\geq N_2$,
\be \label{eq:step 2 estimate}
\PP_N^{(0,\a,\b,0,0)}\lb\#\{j: \lambda_1(j)-\lambda_2(j)\leq r\}\leq \lceil\varepsilon N\rceil \mbox{ } \Bigg{|} \mbox{ } {\substack{ \displaystyle \lambda_1(\lfloor\th_jN\rfloor)-\lambda_2(\lfloor\th_jN\rfloor)\geq \lceil N^{3/4}\rceil, j\in\llb1,d\rrb \\ \displaystyle h_N[\bl](\th_j)\in(u_j,v_j), j\in\llb1,d+1\rrb}}\rb\geq\frac{1}{2}.
\ee 
By conditioning on the values $\lambda_1(\lfloor\theta_j N\rfloor)$ and $\lambda_1(\lfloor\theta_j N\rfloor)$ for $j\in\intint{1,d}$,  it suffices to show the following statement: for $N\geq N_2$ and for any $y_j,z_j\in\ZZ$, $j\in\llb0,d+1\rrb$ such that 
\begin{enumerate}
    \item [$\bullet$] $y_0=0$, 
    \item [$\bullet$] $y_j-z_j\geq\lceil N^{3/4}\rceil$ for $j\in\llb1,d\rrb$,
    \item [$\bullet$] $y_j\in\llb\lfloor u_jN\rfloor,\lfloor v_jN\rfloor\rrb$ for $j\in\llb1,d+1\rrb$,
\end{enumerate}
and there exists $\bl\in\mathcal{TL}_N$ satisfying $\la_1(\lfloor\th_jN\rfloor)=y_j$ and $\la_2(\lfloor\th_jN\rfloor)=z_j$ for $j\in\llb0,d+1\rrb$, we have
\be\label{eq:only need to show step 2}
\PP_N^{(0,\a,\b,0,0)}\lb  \#\{j: \lambda_1(j)-\lambda_2(j)\leq r\}\leq \lceil\varepsilon N\rceil \mbox{ }\Big{|}\mbox{ }\la_1(\lfloor\th_jN\rfloor)=y_j, \la_2(\lfloor\th_jN\rfloor)=z_j,  j\in\llb0,d+1\rrb    \rb\geq\frac{1}{2};
\ee 
then \eqref{eq:step 2 estimate} follows by averaging over the values $\lambda_1(\lfloor\theta_j N\rfloor)$ and $\lambda_1(\lfloor\theta_j N\rfloor)$ for $j\in\intint{1,d}$.
We denote by $h_1:\llb0,N\rrb\rightarrow\RR$ the piece-wise linear function connecting $h_1(\lfloor\th_jN\rfloor)=y_j$ for $j\in\llb0,d+1\rrb$, and likewise by 
$h_2:\llb0,N\rrb\rightarrow\RR$ the piece-wise linear function connecting $h_2(\lfloor\th_jN\rfloor)=z_j$ for $j\in\llb0,d+1\rrb$. In view of $y_j-z_j\geq\lceil N^{3/4}\rceil$ for $j\in\llb1,d\rrb$, there exists $N_3\in\mathbb{N}$ which only depends on $r\in\mathbb{N}$, $\ep>0$ and $\th_j$, $j\in\llb0,d+1\rrb$ (in particular, independent of $y_j$ and $z_j$), such that for $N\geq N_3$, we have
\[
h_1-N^{3/5}> h_2+ N^{3/5}+r+4 \quad \mbox{over} \quad  \llb\lfloor\varepsilon N/2\rfloor,N-\lfloor\varepsilon N/2\rfloor\rrb.
\] 
Hence
\be\label{eq:vegvreba}
\begin{split}
\left\{\#\{j: \lambda_1(j)-\lambda_2(j)\leq r\}\leq \lceil\varepsilon N\rceil\right\}& \supset
\left\{\lambda_1(j)-\lambda_2(j)>r, j\in\llb\lfloor\varepsilon N/2\rfloor,N-\lfloor\varepsilon N/2\rfloor\rrb\right\}\\
& \supset\left\{ \la_1\geq h_1-N^{3/5}, \la_2\leq h_2+ N^{3/5} \right\}.
\end{split}
\ee

For $j\in\llb1,d+1\rrb$, we denote by $h_1^{(j)}:\llb0,\lfloor\th_{j}N\rfloor-\lfloor\th_{j-1}N\rfloor\rrb\rightarrow\RR$ the linear function connecting $h_1^{(j)}(0)=y_{j-1}$ to $h_1^{(j)}(\lfloor\th_{j}N\rfloor-\lfloor\th_{j-1}N\rfloor)=y_j$.
Likewise, we denote by $h_2^{(j)}:\llb0,\lfloor\th_{j}N\rfloor-\lfloor\th_{j-1}N\rfloor\rrb\rightarrow\RR$ the linear function connecting $h_2^{(j)}(0)=z_{j-1}$ to $h_2^{(j)}(\lfloor\th_{j}N\rfloor-\lfloor\th_{j-1}N\rfloor)=z_j$.
We have
\begin{align}
\MoveEqLeft[2]
    \PP_N^{(0,\a,\b,0,0)}\lb  \#\{j: \lambda_1(j)-\lambda_2(j)\leq r\}\leq \lceil\varepsilon N\rceil\mbox{ } \Big{|}\mbox{ }\la_1(\lfloor\th_jN\rfloor)=y_j, \la_2(\lfloor\th_jN\rfloor)=z_j,  j\in\llb0,d+1\rrb    \rb\nonumber\\
    &\geq\PP_N^{(0,\a,\b,0,0)}\lb \la_1\geq h_1-N^{3/5}, \la_2\leq h_2+ N^{3/5} \mbox{ }\Big{|}\mbox{ }\la_1(\lfloor\th_jN\rfloor)=y_j, \la_2(\lfloor\th_jN\rfloor)=z_j,  j\in\llb0,d+1\rrb    \rb\nonumber\\
    &=\prod_{j=1}^{d+1} \PP^{\lfloor\th_{j}N\rfloor-\lfloor\th_{j-1}N\rfloor, (y_{j-1},z_{j-1}), (y_j,z_j)} \lb  \la_1\geq h_1^{(j)}-N^{3/5}, \la_2\leq h_2^{(j)}+ N^{3/5} \rb\nonumber\\
    &\geq\prod_{j=1}^{d+1} \PP^{\lfloor\th_{j}N\rfloor-\lfloor\th_{j-1}N\rfloor, y_{j-1},y_j} \lb  \la_1\geq h_1^{(j)}-N^{3/5} \rb
    \prod_{j=1}^{d+1}\PP^{\lfloor\th_{j}N\rfloor-\lfloor\th_{j-1}N\rfloor, z_{j-1},z_j} \lb\la_2\leq h_2^{(j)}+ N^{3/5}\rb\nonumber\\
    &\geq\lb1-C\exp\lb-c\frac{N^{6/5}}{\lfloor\th_{j}N\rfloor-\lfloor\th_{j-1}N\rfloor}\rb\rb^{2d+2}, \label{eq:bseraer}
\end{align}
where in the first step we use \eqref{eq:vegvreba}, in the second step we use the Gibbs property of random walk bridges (Lemma~\ref{lem:gibbs}), 
in the third step we use Lemma \ref{lem:two paths monotone} and in the fourth step we use Lemma \ref{lem:bound on fluctuation of random walk path}. 
We next choose $N_4\in\mathbb{N}$ such that for $N\geq N_4$, the right hand side of \eqref{eq:bseraer} is greater than $1/2$. Take $N_2=\max(N_3,N_4)$, then for $N\geq N_2$, \eqref{eq:only need to show step 2} holds and hence \eqref{eq:step 2 estimate} holds. \\
	
\noindent\textbf{Conclusion of the proof.}
We next choose $N_5\in\mathbb{N}$ such that for $N\geq N_5$, we have
\be \label{eq:veravrfaferfw}
N^{-2d \lceil N^{3/4}\rceil}\cdot\frac{1}{2}\geq e^{-N^{4/5}}.
\ee
By Step 1, for $N\geq N_1$ we have \eqref{eq:step 1 estimate}. By Step 2, for $N\geq N_2$ we have \eqref{eq:step 2 estimate}. Take $N_0=\max(N_1,N_2,N_5)$. Then for $N\geq N_0$, in view of \eqref{eq:veravrfaferfw}, we have \eqref{eq:key estimate }. We conclude the proof.  
\end{proof}

\section{Proof of the large deviation principle}\label{sec:proof of LDP}

In this section we give the proof of the main result, Theorem~\ref{thm:LDP main thm}. We start by giving the rate function for the finite dimensional LDP in Section~\ref{sec:finite dim LDP}. Section~\ref{sec:standard LPD theory} collects some standard results from LDP theory. The proof of Theorem~\ref{thm:LDP main thm} is given in Section~\ref{subsec:proof of main thm} combining these ingredients with Proposition~\ref{prop:key estimate}.

\subsection{The finite dimensional rate function}\label{sec:finite dim LDP}
We define and analyze the rate functions on $\RR^{d+1}$, which will later serve as the rate functions for the finite dimensional marginal distributions of the height function.

\begin{definition}
Let $\a,\b\geq0$ satisfy $\a\b<1$. Let $d\in\mathbb{N}$ and $0<\th_1<\dots<\th_d<\th_{d+1}=1$. We consider the function $\Iabth:\RR^{d+1}\rightarrow[0,\infty]$ defined as
\bestar 
\Iabth(x_1,\dots,x_{d+1}):=\inf_{f\in C_0([0,1],\RR), f(\th_j)=x_j, j\in\llb1,d+1\rrb}\Iab(f), 
\eestar 
where $\Iab:C_0\lb[0,1],\RR\rb\rightarrow[0,\infty]$ is defined in the statement of Theorem \ref{thm:LDP main thm}. 
Since $\Iab(f)<\infty$ if and only if $f\in\mathcal{AC}_0$ and $0\leq f'\leq1$, we have that $\Iabth(x_1,\dots,x_{d+1})<\infty$ if and only if
\be \label{eq:conditions cone}
0\leq x_{j}-x_{j-1}\leq\th_{j}-\th_{j-1} \quad \mbox{for}\quad j\in\llb1,d+1\rrb,
\ee 
where we denote $\th_0=x_0=0$. The conditions \eqref{eq:conditions cone} specify a subset $\cone_{\th_1,\dots,\th_{d+1}}\subset\RR^{d+1}$. 
\end{definition}
\begin{remark}
    By the alternative formula of $\Iab$ given in \cite{bryc2024two} (see Remark \ref{rmk:alternative formula rate function}), we have
    \be \label{eq:finite dimensional alternative formula}
    I^{(\a,\b)}_{\th_1,\dots,\th_{d+1}}(x_1,\dots,x_d)= \inf_{ \substack{ f,g\in\mathcal{AC}_0, \mbox{  } 0\leq f',g'\leq 1 \mbox{ a.s.} \\ f(\th_j)=x_j,\mbox{ }j\in\llb1,d+1\rrb }}\Iab(f,g),
    \ee 
    where we recall 
    \be \label{eq:alternative formula Ifg again}
    \Iab(f,g)= \int_0^1 \bigl(H(f'(x)) +H(g'(x))\bigr) \d x
    + \log(\a\b) \min_{0\leq x\leq 1}(f(x)-g(x)) -\log (\b)(f(1)-g(1)) - \log (J(\a,\b)).
    \ee 
\end{remark}

The following result will be useful later in the proof of the large deviation principle.

\begin{lemma}\label{lem:rate function at boundary}
Let $\a,\b\geq0$ satisfy $\a\b<1$. Let $d\in\mathbb{N}$ and $0<\th_1<\dots<\th_d<\th_{d+1}=1$.
    Then for any open subset $U\subset\RR^{d+1}$ we have
    \begin{equation*}\label{eq:rate function at boundary}
        \inf_{x\in U}I^{(\a,\b)}_{\th_1,\dots,\th_{d+1}}(x)=\inf_{x\in U\cap\cone^{\circ}}I^{(\a,\b)}_{\th_1,\dots,\th_{d+1}}(x),
    \end{equation*}
where $\cone^{\circ}$ denotes the interior of the cone $\cone=\cone_{\th_1,\dots,\th_{d+1}}$, i.e.,
\be \label{eq:interior of cone}
\cone^{\circ}=\left\{(x_1,\dots,x_{d+1})\in\RR^{d+1}: 0< x_{j}-x_{j-1} < \th_{j}-\th_{j-1}, j\in\llb1,d+1\rrb\right\}.
\ee 
\end{lemma}

\begin{proof} 
    We can assume $U\setminus\lb U\cap\cone^{\circ}\rb\neq\emptyset$ without loss of generality.
    By \eqref{eq:finite dimensional alternative formula}, we only need to show 
    \begin{equation}\label{eq:only need to prove rate function at boundary}
        \inf_{ \substack{ f,g\in\mathcal{AC}_0, \mbox{  } 0\leq f',g'\leq 1 \text{ a.e.} \\ (f(\th_1),\dots,f(\th_{d+1}))\in U }} I^{(\a,\b)}(f,g)
        =\inf_{ \substack{ f,g\in\mathcal{AC}_0, \mbox{  } 0\leq f',g'\leq 1 \text{ a.e.} \\ (f(\th_1),\dots,f(\th_{d+1}))\in U\cap\cone^{\circ} }} I^{(\a,\b)}(f,g).
    \end{equation} 
    We only need to consider functions $f$ such that $f\in\mathcal{AC}_0$, $0\leq f'\leq1$ almost everywhere, and it holds that $(f(\th_1),\dots,f(\th_{d+1}))\in U\setminus(U\cap\cone^{\circ})$. We define
    \bestar 
    \operatorname{Ind}_0(f):=\{j\in\llb1,d+1\rrb: f(\th_j)=f(\th_{j-1})\},\quad
    \operatorname{Ind}_1(f):=\{j\in\llb1,d+1\rrb: f(\th_j)-f(\th_{j-1})=\th_j-\th_{j-1} \}.
    \eestar 
    Then $\emptyset\neq\operatorname{Ind}_0(f)\cup\operatorname{Ind}_1(f)\subseteq\llb1,d+1\rrb$ and $\operatorname{Ind}_0(f)\cap\operatorname{Ind}_1(f)=\emptyset$. Note that if $j\in\operatorname{Ind}_0(f)$, then $f'(x) = 0$ for almost every $x\in[\theta_{j-1}, \theta_j]$, and if $j\in\operatorname{Ind}_1(f)$, then $f'(x) = 1$ for almost every $x\in[\theta_{j-1}, \theta_j]$. With this observation, for $\ep>0$ we consider the function $f_{\ep}(x)$ on $x\in[0,1]$ defined by
    \begin{multline*}
    f_{\ep}(x):=f(x)+\sum_{j\in\operatorname{Ind}_0(f)}\lb(x-\th_{j-1})\one_{x\in[\th_{j-1},\th_{j-1}+\ep]}+\ep\one_{x>\th_{j-1}+\ep}\rb\\-\sum_{j\in\operatorname{Ind}_1(f)}\lb(x-\th_{j-1})\one_{x\in[\th_{j-1},\th_{j-1}+\ep]}+\ep\one_{x>\th_{j-1}+\ep}\rb.
    \end{multline*}
    For 
    $\ep>0$ sufficiently small, we have $f_{\ep}\in\mathcal{AC}_0$,  $0\leq f_{\ep}'\leq1$ almost everywhere and $(f_{\ep}(\th_1),\dots,f_{\ep}(\th_{d+1}))\in U\cap\cone^{\circ}$. Since $H(1)=H(0)=0$ and whenever $f'(x)\neq f'_\varepsilon(x)$ it holds that one equals $1$ and the other $0$, we have $H(f'(x))=H(f'_{\ep}(x))$ almost everywhere. Therefore, in view of the formula \eqref{eq:alternative formula Ifg again} for $I(f,g)$, for any functions $g$ such that $g\in\mathcal{AC}_0$ and $0\leq g'\leq1$ almost everywhere, 
    \[
    \lim_{\ep\rightarrow0+}I(f_{\ep},g)=I(f,g).
    \]
    We conclude the proof of \eqref{eq:only need to prove rate function at boundary} and hence the proof of the result.
\end{proof}

\subsection{Some standard results in large deviation theory}\label{sec:standard LPD theory}
In this subsection, we record several standard results from the theory of large deviations in \cite{dembo2009large}, which will be useful later in our proof.
\begin{theorem}[Contraction principle, c.f. Theorem 4.2.1 in \cite{dembo2009large}]\label{thm:contraction principle}
    Let $\mathcal{X}$ and $\mathcal{Y}$ be Hausdorff topological spaces and $f:\mathcal{X}\rightarrow\mathcal{Y}$ be a continuous function. Consider a good rate function $I:\mathcal{X}\rightarrow[0,\infty]$.
    \begin{enumerate}
        \item [(1)] For each $y\in\mathcal{Y}$, define
        \[
        I'(y):=\inf\{I(x):x\in\mathcal{X},y=f(x)\}.
        \]
        Then $I'$ is a good rate function on $\mathcal{Y}$, where as usual the infimum over the empty set is taken as $\infty$.
        \item [(2)] If $I$ controls the LDP associated with a family of probability measures $\{\PP_N\}_{N=1}^{\infty}$ on $\mathcal{X}$, then $I'$ controls the LDP associated with the family of probability measures $\{\PP_N\circ f^{-1}\}_{N=1}^{\infty}$ on $\mathcal{Y}$.
    \end{enumerate}
\end{theorem}

\begin{theorem}[Dawson--Gärtner Theorem, c.f. Theorem 4.6.1 in \cite{dembo2009large}]\label{thm:Dawson–Gärtner}
    Let $\mathcal{X}=\lim_{\leftarrow}\mathcal{M}_j$ be the projective limit of a projective system $\lb\mathcal{M}_j,p_{ij}\rb_{i\leq j\in J}$ of Hausdorff topological spaces $\lb\mathcal{M}_j\rb_{j\in J}$ with continuous maps $p_{ij}:\mathcal{M}_j\rightarrow\mathcal{M}_i$. Let $p_j:\mathcal{X}\rightarrow\mathcal{M}_j$, $j\in J$ be the canonical projections.
    Let $\{X_N\}_{N=1}^{\infty}$ be a sequence of random variables taking values in $\mathcal{X}$, such that for any $j\in J$, the sequence of random variables $\{p_j\circ X_N\}_{N=1}^{\infty}$ taking values in $\mathcal{M}_j$ satisfies the LDP with the good rate function $I_j(\cdot)$. Then $\{X_N\}_{N=1}^{\infty}$ satisfies the LDP with the good rate function:
    $$I(\mathbf{x})=\sup_{j\in J}I_j(p_j(\mathbf{x})),\quad \mathbf{x}\in\mathcal{X}.$$
\end{theorem}

\begin{lemma}[c.f. Lemma 4.1.4 in \cite{dembo2009large}]\label{lem:uniqueness}
    A sequence of random variables taking values on a regular topological space (in particular, a metric space) can have at most one rate function associated with its LDP.
\end{lemma}

\begin{lemma}[c.f. Lemma 4.1.5 (b) in  \cite{dembo2009large}]\label{lem:Lemma 4.1.5(b)}
    Let $\{X_N\}_{N=1}^{\infty}$ be a sequence of random variables taking values in a topological space $\mathcal{X}$. Let $\mathcal{E}\subset\mathcal{X}$ be a Borel measurable subset such that $X_N$ almost surely take values in $\mathcal{E}$, for each $N=1,2,\dots$. If $\{X_N\}_{N=1}^{\infty}$ satisfies the LDP in $\mathcal{X}$ with the good rate function $I$ and $I=\infty$ on $\mathcal{X}\setminus\mathcal{E}$, then the same LDP holds on $\mathcal{E}$.
\end{lemma}
 
\begin{lemma}[c.f. Corollary 4.2.6 in \cite{dembo2009large}]\label{lem:Corollary 4.2.6}
Let $\{X_N\}_{N=1}^{\infty}$ be an exponentially tight sequence of random variables taking values in $\mathcal{X}$ equipped with the topology $\mathbf{t}_1$ (see \cite[page 8]{dembo2009large} for the definition of exponential tightness). If $\{X_N\}_{N=1}^{\infty}$ satisfies the LDP with respect to a Hausdorff topology $\mathbf{t}_2$ on $\mathcal{X}$ that is coarser than $\mathbf{t}_1$, then the same LDP holds with respect to the topology $\mathbf{t}_1$.
\end{lemma}

\subsection{Proof of Theorem \ref{thm:LDP main thm}}\label{subsec:proof of main thm}

\begin{proof}[Proof of Theorem \ref{thm:LDP main thm}]
Let $\a,\b,\c,\d,q$ satisfy \eqref{eq:conditions qABCD} and $\a\b<1$.  We know from \cite{bryc2024two} that the main Theorem \ref{thm:LDP main thm} holds in the special case $\c=\d=q=0$. We now show that it holds true in the general case. By the result in \cite{bryc2024two} for $\c=\d=q=0$ and the contraction principle (Theorem \ref{thm:contraction principle}), for any fixed $0<\th_1<\dots<\th_{d+1}=1$, the re-scaled height function $(h_N(\th_1),\dots,h_N(\th_{d+1}))\in\RR^{d+1}$ under the stationary measure $\mu_N^{(0,\a,\b,0,0)}$ satisfies the large deviation principle with good rate function $\Iabth$. In view of Definition \ref{def:LDP} and the two-layer representation (Theorem \ref{thm:two layer representation}), for any closed subset $C\subset\RR^{d+1}$, 
\be \label{eq:closed fdd rate function 0}
\limsup_{N\to\infty}\frac{1}{N} \log \PP_N^{(0,\a,\b,0,0)} ((h_N[\bl](\th_1),\dots,h_N[\bl](\th_{d+1}))\in  C)\leq -\inf_{x\in C} I^{(\a,\b)}_{\th_1,\dots,\th_{d+1}}(x),
\ee
and for any open subset $U\subset\RR^{d+1}$,
\be\label{eq:open fdd rate function 0}
\liminf_{N\to\infty}\frac{1}{N} \log \PP_N^{(0,\a,\b,0,0)} ((h_N[\bl](\th_1),\dots,h_N[\bl](\th_{d+1})) \in  U)\geq -\inf_{x\in U} I^{(\a,\b)}_{\th_1,\dots,\th_{d+1}}(x).
\ee 

We split the rest of the proof into two steps.\\
	
\noindent\textbf{Step 1.} In this step we show that, for any fixed $0<\th_1<\dots<\th_{d+1}=1$, under the stationary measure $\mu_N^{(q,\a,\b,\c,\d)}$, the re-scaled height function $(h_N(\th_1),\dots,h_N(\th_{d+1}))\in\RR^{d+1}$ of open ASEP satisfies the large deviation principle with the same good rate function $\Iabth$. 

Recall from \eqref{eq:upper bound weight} that $\wt_N^{(q,\a,\b,\c,\d)}(\bl)\leq\wt_N^{(0,\a,\b,0,0)}(\bl)/(q;q)_{\infty}$ for any $\bl\in\mathcal{TL}_N$ and from Lemma \ref{lem:log asy of normalization constant} that the limit of $\frac{1}{N}\log Z_N^{(q,\a,\b,\c,\d)}$ does not depend on $q$. So by \eqref{eq:closed fdd rate function 0}, for any closed subset $C\subset\RR^{d+1}$,
\begin{multline}\label{eq:closed conclude fdd}
\limsup_{N\to\infty}\frac{1}{N} \log \PP_N^{(q,\a,\b,\c,\d)} ((h_N[\bl](\th_1),\dots,h_N[\bl](\th_{d+1}))\in  C)\leq\\
\limsup_{N\to\infty}\frac{1}{N} \log \PP_N^{(0,\a,\b,0,0)} ((h_N[\bl](\th_1),\dots,h_N[\bl](\th_{d+1}))\in  C)
\leq-\inf_{x\in C} I^{(\a,\b)}_{\th_1,\dots,\th_{d+1}}(x).
\end{multline} 

For any open subset $U\subset\RR^{d+1}$, we want to show 
\be\label{eq:only need to show fdd LDP open q}
\liminf_{N\to\infty}\frac{1}{N} \log \PP_N^{(q,\a,\b,\c,\d)} ((h_N[\bl](\th_1),\dots,h_N[\bl](\th_{d+1})) \in  U)\geq -\inf_{x\in U} I^{(\a,\b)}_{\th_1,\dots,\th_{d+1}}(x).
\ee 
In view of Lemma \ref{lem:rate function at boundary}, we only need to show \eqref{eq:only need to show fdd LDP open q} for open subsets $U\subset\cone^\circ$, where $\cone^\circ$ is given by \eqref{eq:interior of cone}. One can observe that if $U=\cup_{i=1}^{\infty}U_i$ and \eqref{eq:only need to show fdd LDP open q} holds for each $U_i$, then it also holds true for $U$. Note that any open subset of $\cone^\circ$ is a countable union of open rectangles of the form $(u_1,v_1)\times\dots\times(u_{d+1},v_{d+1})$, each of which has positive distance to the boundary of $\cone\subset\RR^d$. Such rectangles satisfy
\be 
0< u_{j}-v_{j-1} < v_{j}-u_{j-1} < \th_{j}-\th_{j-1}, j\in\llb1,d+1\rrb.
\ee 
In particular, the assumptions of  Corollary \ref{cor:key estimate} are satisfied. Combining Corollary \ref{cor:key estimate} and Lemma~\ref{lem:log asy of normalization constant},  
\[
\lim_{N\to\infty}\frac{1}{N} \log \frac{\PP_N^{(q,\a,\b,\c,\d)} ((h_N[\bl](\th_1),\dots,h_N[\bl](\th_{d+1})) \in  U)}{\PP_N^{(0,\a,\b,0,0)} ((h_N[\bl](\th_1),\dots,h_N[\bl](\th_{d+1})) \in  U)}=0
\]
for any open rectangle $U=(u_1,v_1)\times\dots\times(u_{d+1},v_{d+1})$ in $\cone^{\circ}$ which has positive distance to its boundary.  
Combining with \eqref{eq:open fdd rate function 0}, we know that \eqref{eq:only need to show fdd LDP open q} holds for such open rectangles $U$. By the arguments above, we conclude that \eqref{eq:only need to show fdd LDP open q} holds for any open subset $U\subset\RR^d$.

Combining \eqref{eq:closed conclude fdd} and \eqref{eq:only need to show fdd LDP open q}, we conclude that under the stationary measure $\mu_N^{(q,\a,\b,\c,\d)}$, the  height function $(h_N(\th_1),\dots,h_N(\th_{d+1}))\in\RR^{d+1}$ satisfies the large deviation principle with good rate function $\Iabth$.\\
	
\noindent\textbf{Step 2.} In this step we use the Dawson--Gärtner Theorem to enhance the large deviation principle from the finite dimensional marginals $(h_N(\th_1),\dots,h_N(\th_{d+1}))$ of $h_N$ under $\mu_N^{(q,\a,\b,\c,\d)}$ to the level of the path space, thus concluding the proof. This argument is similar to \cite[Proof of Theorem 5.1.2]{dembo2009large}.

Define $\mathcal{X}$ to be the space of (arbitrary) functions $f:[0,1]\rightarrow\RR$ such that $f(0)=0$, equipped with the topology of point-wise convergence, i.e., the topology generated by 
$V_{t,x,\delta}:=\{g:[0,1]\rightarrow\RR: g(0)=0, |g(t)-x|<\delta\}$ for all $t\in(0,1]$, $x\in\RR$ and $\delta>0$. 
For any $d\in\mathbb{N}$ and $0<\th_1<\dots<\th_{d+1}=1$, there is a natural projection $\mathcal{X}\rightarrow\RR^{d+1}$ defined by $f\mapsto(f(\th_1),\dots,f(\th_{d+1}))$. Observe that they form a projective system with the partial ordering given by inclusion over the index set
\[\cup_{d=1}^{\infty}\{(\th_1,\dots,\th_{d+1}): 0<\th_1<\dots<\th_{d+1}=1\},\]
and that $\mathcal{X}$ is the projective limit. By the Dawson--Gärtner Theorem (Theorem \ref{thm:Dawson–Gärtner}), the large deviation principle holds for the height profiles $h_N$ under the stationary measure $\mu_N^{(q,\a,\b,\c,\d)}$, considered as random variables taking values in $\mathcal{X}$, with good rate function $\widehat{I}^{(\a,\b)}:\mathcal{X}\rightarrow[0,\infty]$ given by
\be 
\widehat{I}^{(\a,\b)}(f):=\sup_{0<\th_1<\dots<\th_{d+1}=1}\Iabth\lb f(\th_1),\dots,f(\th_{d+1})\rb.
\ee

We denote the subset of (non-strictly) increasing, $1$-Lipschitz continuous functions by $\mathcal{Y}\subset\mathcal{X}$. Then $\mathcal{Y}$ can be seen as a closed subset in the (Hausdorff) topology of $\mathcal{X}$. 
The height profiles $\{h_N\}_{N=1}^{\infty}$ are random variables taking values in $\mathcal{Y}$. Using \eqref{eq:def of LDP first ineq} for $U=\mathcal{X}\setminus\mathcal{Y}$ we conclude that $\widehat{I}^{(\a,\b)}=\infty$ on $\mathcal{X}\setminus\mathcal{Y}$. In view of $\mathcal{Y}\subset C_0\lb[0,1],\RR\rb\subset\mathcal{X}$ and using Lemma \ref{lem:Lemma 4.1.5(b)}, the large deviation principle holds on $C_0\lb[0,1],\RR\rb$ with good rate function $\widehat{I}^{(\a,\b)}$, but with the topology on $C_0\lb[0,1],\RR\rb$ induced from $\mathcal{X}$, which is coarser than the supremum norm topology. 
By the Arzelà–Ascoli theorem, $\mathcal{Y}$ is a compact subset in the supremum norm topology. Since the height profiles take values in $\mathcal{Y}$, they induce an exponentially tight sequence of measures
(see \cite[page 8]{dembo2009large} for the definition of exponential tightness). Then by Lemma \ref{lem:Corollary 4.2.6}, the large deviation principle with good rate function $\widehat{I}^{(\a,\b)}$ can be strengthened to  $C_0\lb[0,1],\RR\rb$ with the supremum norm topology. Since $\widehat{I}^{(\a,\b)}$ only depends on $\a$ and $\b$, in view of \cite{bryc2024two} that Theorem \ref{thm:LDP main thm} holds for $\c=\d=q=0$ and the uniqueness of rate function (Lemma \ref{lem:uniqueness}), we have $\widehat{I}^{(\a,\b)}=\Iab$, where $\Iab$ is the rate function defined in the statement of Theorem \ref{thm:LDP main thm}. The proof is concluded. 
\end{proof}

\bibliographystyle{goodbibtexstyle}
\bibliography{twolayerlimitshape}

\end{document}